\documentclass[a4paper,11pt]{article}

\usepackage{amssymb}
\usepackage{graphicx}
\usepackage{tikz}
\usepackage{pgfplots}
\pgfplotsset{compat=1.11}
\usepackage{mathrsfs}
\usepackage{amsmath}
\usepackage{amsthm}
\usepackage{enumerate}
\usepackage{color}
\usepackage{verbatim}
\usepackage{todonotes}
\usepackage{mathtools}

\usepackage[margin=2.3cm, includefoot, footskip=30pt]
{geometry}

\usepackage{array}
\newcolumntype{e}{>{\displaystyle}r @{\,} >{\displaystyle}c @{\,} >{\displaystyle}l}

\usepackage{amssymb,amsopn,amscd}

\usepackage[colorlinks]{hyperref}
\hypersetup{final}
\usepackage{xcolor}
\hypersetup{
    colorlinks,
    linkcolor={blue!40!black},
    citecolor={blue!40!black},
    urlcolor={blue!80!black}
}
\usepackage{nameref}

\usepackage[figure,table]{hypcap} 

\theoremstyle{plain}
\newtheorem{Theorem}{Theorem}[section]
\newtheorem{Corollary}[Theorem]{Corollary}
\newtheorem{Lemma}[Theorem]{Lemma}
\newtheorem{Proposition}[Theorem]{Proposition}

\theoremstyle{definition}

\newtheorem{Definition}[Theorem]{Definition}
\newtheorem{Remark}[Theorem]{Remark}

\newcommand{\E}{\mathbb{E}}
\newcommand{\R}{\mathbb{R}}
\newcommand{\N}{\mathbb{N}}
\newcommand{\Z}{\mathbb{Z}}

\renewcommand{\P}{\mathbb{P}}
\newcommand{\m}{\mathbf{m}}

\newcommand{\Proba}{\mathbb{P}}
\newcommand{\range}[2]{[\![#1,#2]\!]}

\numberwithin{equation}{section}

  \newcounter{constant}

\begin{document}

\title{First Passage Percolation with Recovery}

\author{Elisabetta Candellero\footnote{elisabetta.candellero@uniroma3.it;  University of Roma Tre, Rome} 
\and 
Tom Garcia-Sanchez\footnote{tom.garcia-sanchez@ens.psl.eu -- \'Ecole Normale Sup{\'e}rieure Paris (Master student)}
}

\maketitle

\begin{abstract}
First passage percolation with recovery is a process aimed at modeling the spread of epidemics.
On a graph $G$ place a red particle at a reference vertex $o$ and colorless particles (seeds) at all other vertices.
The red particle starts spreading a \emph{red first passage percolation} of rate $1$, while all seeds are dormant.
As soon as a seed is reached by the process, it turns red and starts spreading {red first passage percolation}. 
All vertices are equipped with independent exponential  clocks ringing at rate $\gamma>0$, when a clock rings the corresponding \emph{red vertex turns black}.
For $t\geq 0$, let $H_t$ and $M_t$ denote the size of the longest red path and of the largest red cluster present at time $t$. 
If $G$ is the semi-line, then for all $\gamma>0$ almost surely $\limsup_{t}\frac{H_t\log\log t}{\log t}=1 $ and $\liminf_{t}H_t=0$.
In contrast, if $G$ is an infinite Galton-Watson tree with offspring mean $\m>1$ then, for all $\gamma>0$, 
almost surely $\liminf_{t}\frac{H_t\log t}{t}\geq\m-1 $ and $\liminf_{t}\frac{M_t\log\log t}{t}\geq \m-1$, while $\limsup_{t} \frac{M_t}{e^{c t}}\leq 1$, for all $c>\m -1$.
Also, almost surely as $t\to \infty$, for all $\gamma>0$ $H_t$ is of order at most $t$.
Furthermore, if we restrict our attention to bounded-degree graphs, then for any $\varepsilon>0$ there is a critical value $\gamma_c>0$ so that for all $\gamma>\gamma_c$, almost surely $\limsup_{t}\frac{M_t}{t}\leq \varepsilon $.
\end{abstract}

\textbf{Keywords}:
First passage percolation, first passage percolation in hostile environment, competition, recovery, branching processes, Galton-Watson trees, epidemics

\section{Introduction}\label{sect:introduction}

In this work we study a natural process which can be seen as an epidemic process where individuals can recover, but they can still transmit the disease if they are close to a susceptible individual.
Intuitively, the process is defined as follows.
On a graph $G$ (which we can imagine infinite, locally finite and connected) place a \emph{red particle} at some reference vertex $o$ and \emph{colorless particles} at all other vertices.
At time $0$ the red particle starts spreading a \emph{red} first passage percolation (FPP) of rate $1$, which we can think of a disease starting at a single source and spreading throughout a network.
As soon as a seed is reached by the process, it instantly turns red and starts spreading red FPP.
All vertices are equipped with independent exponential clocks that ring at a given rate $\gamma>0$; as soon as a clock rings, the corresponding vertex \emph{turns black}.
This fact does not inhibit the red process; in fact, active seeds still spread red FPP, even if their vertices turned black.


This natural model is motivated by the following interpretation.
Suppose that vertices of $G$ represent individuals in a community where red FPP represents an infectious disease. 
Then, red vertices represent infected individuals and black vertices those that have recovered from the disease.
The \emph{recovery rate} clearly represents the rate at which a sick individual becomes healthy.
Here we always assume that a recovered individual will not be able to be infected again.
However, recovered individuals can spread the infection to neighboring sites (those hosting seeds).

We shall study the asymptotic behavior of the size of the longest red path and of the largest red cluster in two cases, namely when $G$ is the semi-line $\N$, and when $G$ is a supercritical Galton-Watson tree (including the trivial case of a homogeneous tree).

To the best of our knowledge, this is the very first time that such a process is analyzed, and we hope that this work will trigger further investigation in this direction.

\subsection{Main results}
In the following we shall refer to ``\emph{red} vertices'' as those whose seeds have been activated by red FPP but their corresponding \emph{recovery clock} hasn't rung yet.
Similarly, ``\emph{black} vertices'' are the ones already reached by red FPP whose recovery clock has already rung.

For all $t\in\R$ let $H_t$ be the \emph{size of the longest (oriented) path consisting of red vertices} present at time $t$, and $M_t$ be the \emph{size of the largest red cluster} present at time $t$.
The present work focuses on the asymptotic behavior of $H_t$ and $M_t$ as $t\to \infty$, when $G$ is a tree (deterministic or random).
More precisely, we show the following results.

\begin{Theorem}
\label{thm_critical}
When $G$ is the infinite semi-line $\N$ then, for all recovery rates $\gamma>0$, 
\begin{itemize}
\item[(i)] $\P \left ( \limsup_{t\rightarrow+\infty}\frac{H_t\log\log t}{\log t}=1 \right )=1$;
\item[(ii)] $\P \left ( \liminf_{t\rightarrow+\infty}H_t=0\right ) =1$.
\end{itemize}
\end{Theorem}
Roughly speaking, when $G=\N$ then the ``largest'' red component gets of order $\frac{\log t}{\log \log t}$, when $t$ is large.
However, even if this event occurs infinitely often, it is not true that anything close to this size is maintained, as the second item shows.
In fact, almost surely, there are arbitrary large times so that there is no red path in the graph, because all infected vertices have recovered.
In contrast, the next theorem shows that whenever $G$ is a super-critical Galton-Watson tree with finite mean then eventually there will be non-trivial red paths (and red clusters) in the graph.
More precisely, for all large times $t$, the longest red path is of order at least $\frac{t}{\log t}$ and the largest red cluster is of order at least $\frac{t}{\log \log t}$.
The statement includes as well the ``degenerate'' case of a complete infinite $d$-ary tree with branching number  $d\geq 2$. 

\begin{Theorem}
\label{thm_supercritical}
If $G$ is a supercritical Galton-Watson tree whose offspring distribution has finite mean $\m>1$, then for all $\gamma>0$, 
\begin{itemize}
\item[(i)] $\displaystyle \P \left ( \liminf_{t\rightarrow+\infty}\frac{H_t\log t}{t}\geq\m-1 \mid G\text{ is infinite}\right ) =1 $;
\item[(ii)] $\displaystyle \P \left ( \liminf_{t\rightarrow+\infty}\frac{M_t\log\log t}{t}\geq \m-1 \mid G\text{ is infinite} \right )=1$.
\end{itemize}
\end{Theorem}

As a consequence of Theorem \ref{thm_supercritical} we obtain the following result, where $\#A_t$ denotes the number of vertices reached by the process (either red or black) by time $t\geq 0$.

\begin{Corollary}
\label{coro_lb_gw_at}
Under the hypotheses of Theorem \ref{thm_supercritical} we have 
\[
\P \left ( \liminf_{t\rightarrow+\infty}\frac{H_t\log\log(\#A_t)}{\log(\#A_t)}\geq1 \mid G \text{ is infinite}\right )=1.
\]
\end{Corollary}

A consequence of 
\cite{Ariascastro} gives the following bounds on the limsup.
In particular, the first two items follow from bounding the size of the cluster $\# A_t$.

\begin{Proposition}\label{prop:blue}
Under the hypotheses of Theorem \ref{thm_supercritical}, for all $\delta>0$,
\[
\P \left ( \limsup_{t\to+\infty} \frac{M_t}{e^{(\m -1+\delta)t}}=0  \right )=1.
\]
Moreover, there exists a constant $\bar{c}>1$ such that
\[
\P \left ( \limsup_{t\to+\infty} \frac{ H_{t}}{t} \leq \bar{c}  \right ) =1.
\]
Furthermore, when $G$ has bounded degree, then for any $\varepsilon>0$  there is a critical value $\gamma_c>0$ so that for all $\gamma>\gamma_c$, $\P \left ( \limsup_t \frac{M_{t}}{t}\leq \varepsilon \right )=1$.
\end{Proposition}

\subsection{Strategy of the proofs}
Below we outline the proof strategy behind Theorems \ref{thm_critical}, \ref{thm_supercritical} and Proposition \ref{prop:blue}.

\textbf{Idea of the proof of Theorem \ref{thm_critical}.}
The starting observation is the fact that at any time the longest red path is very likely to be at the boundary of the set of occupied vertices. 
It is therefore possible to find the distribution of such a cluster; obtaining the first statement as a consequence.
For the second part, by coupling the process with a suitable random walk on the non-negative integers, we show that almost surely there are arbitrary large times so that there is no red vertex in the graph.

\textbf{Idea of the proof of Theorem \ref{thm_supercritical}.}
As in the previous case, for every $t$, large red clusters tend to be close to the boundary of the occupied region $A_t$.
We show that if this boundary is large, then the probability that all clusters at time $t$ are ``too small'' vanishes quickly.
This will follow from self-similarity (in a distributional sense) of the Galton-Watson tree, together with the fact that if $G$ is infinite, then it grows exponentially fast.
The proof proceeds in steps that take care of $(H_t)_t$ and $(M_t)_t$ at the same time.
For all $t\geq 0$ let $Q_t$ be either $H_t$ or $M_t$, i.e., the size of either the longest red path or the largest red cluster at time $t$.

Fix $t\geq 0$ and consider a vertex $v$ on the \emph{external} boundary of $A_t$.
Clearly, $v$ has not yet been reached by the process, and the difference between its reaching time and $t$ is distributed as an Exp($1$) random variable.
This fact leads us to consider the auxiliary quantity $\P(Q_{1-T}\leq m)$, for all $m\geq 1$, where $T$ is an independent Exp($1$) random variable.
Subsequently, we show that for all $t\geq 1$ it is possible to compare $\P(Q_t\leq m)$ with $\P(Q_{1-T}\leq m)$.
Finally, we show that if $\P(Q_{1-T}\leq m)$ decreases fast enough in $m$, then the liminf has the sought expression; this is the most delicate part of the proof.
We deduce the result by letting $m$ be a suitable function of $t$.

\textbf{Idea of the proof of Proposition \ref{prop:blue}.}
An asymptotic analysis of the growth of the cluster implies the first part of the statement.
The second part follows from a direct calculation that gives an exponential upper bound on the probability that $A_t$ becomes atypically large as $t$ grows.
The last part follows from a coupling with a suitable Bernoulli percolation process and then applying a result of \cite{Ariascastro}.
More precisely, if $\gamma$ is large enough then the corresponding  Bernoulli percolation has parameter $p$ so small, that the cluster of open vertices (starting at the root) up to generation $\ell$ has volume of order at most $\ell$, when $\ell $ goes to infinity.

%

\paragraph{Structure of the paper.}
In Section \ref{sect:def_process} we introduce the main notation and give a formal definition of the process; then we review some related work, in order to put the proposed model into context.
Then we naturally split the remaining work into three parts: Sections \ref{sect:line}, \ref{sect:GWtree}, and \ref{section:blue}, devoted to the proofs of Theorem \ref{thm_critical}, Theorem \ref{thm_supercritical} (and Corollary \ref{coro_lb_gw_at}) and Proposition \ref{prop:blue}, respectively.

\section{Definition of the process}\label{sect:def_process}
In this note we focus on \emph{FPP with recovery}, which we formally define as follows.
On a graph $G=(V,E)$ fix a reference vertex $o\in V$ that we shall refer to as \emph{the origin} (or \emph{the root}, when $G$ is a tree).
At time $0$ we place a \emph{red} particle at $o$ and a colorless \emph{seed} on all other vertices $x\neq o$.
On every edge $(u,v)\in E$ we place an exponential random variable of rate $1$ (independent of everything else) denoted by $T_{(u,v)}$, whereas on all vertices we place independent random variables $\{\mathbf{C}_x \}_{x\in V}$ distributed as exponentials of rate $\gamma>0$, independent of the rest of the process.
Later on we shall refer to $T_{(u,v)}$ as the \emph{passage time} between $u$ and $v$, and to $\mathbf{C}_x$ as the \emph{recovery time} (or \emph{recovery clock}) of vertex $x$.
As soon as the seed at a vertex $v$ is active (hence $v$ is red), it will take some random time before $v$ recovers and turns black; such time is exactly $\mathbf{C}_v$.
Note that once a vertex turns black, then it will never turn red again.

\subsection{Notation}

We start by setting some notation, which is inherited by classical FPP; for a recent overview on FPP see e.g.\ \cite{Auffinger}.


For all $v\in V$ we define $\tau_v$ to be the \emph{reaching time} of $v$ from the origin, i.e.
\begin{equation}\label{eq:def-reachingtime}
\tau_v\coloneqq  
\inf_{\pi}
\ \sum_{j=1}^{|\pi|-1} T_{(\pi_i, \pi_{i+1})}
\end{equation}
where the infimum is taken over all connected paths $\pi=(o=\pi_1, \pi_2, ..., \pi_{|\pi|-1}, \pi_{|\pi|}=v)$ joining $o$ to $v$, with the convention $\inf\emptyset \coloneqq +\infty$.
For all $t\in\R$, define $A_t$ to be the set of all vertices that have been reached by the process by time $t$, i.e.,
\[
A_t\coloneqq \{v\in V \ : \ \tau_v\leq t\}.
\]
Recall that $(\mathbf{C}_v)_{v\in V}$ are i.i.d.\ random variables distributed as $\mathrm{Exp}(\gamma)$, representing the recovery clock of each vertex.
Then, for $t\geq 0$ let 
\begin{equation}\label{eq:def-R_t}
R_t\coloneqq \{v\in A_t \ : \ \tau_v\leq t<\tau_v+\mathbf{C}_v\},
\end{equation}
i.e., the set of all vertices that are red at time $t$.

Finally, recall that for all $t \geq 0$ we define $H_t$ to be the \emph{size (i.e.\ number of vertices) of the longest (oriented) red path} present at time $t$  and $M_t$ to be the \emph{size of the largest red cluster} present at time $t$.
A graphical representation is given in Figure \ref{fig:model-scheme}.

Throughout, we let $\N$ be the set of non-negative integers, while we set $\N^*\coloneqq \N\setminus\{0\}$.
Moreover, for any pair of non-negative integers $a,b$ with $a\leq b$ we denote $\range{a}{b}\coloneqq \N\cap[a, b]$.
Intuitively, $\range{a}{b}$ denotes the (ordered) set of integers between $a$ and $b$.
For every set of vertices $X\subset V$ we let $\# X$ denote its cardinality. 
Finally, whenever for some functions $f,g$ we write that for $t$ large we have $f(t)\sim g(t)$, we mean that $\frac{f(t)}{g(t)}\to 1$, as $t\to+\infty$.

\begin{figure}[h!]
    \centering
    \includegraphics[height=4.5cm]{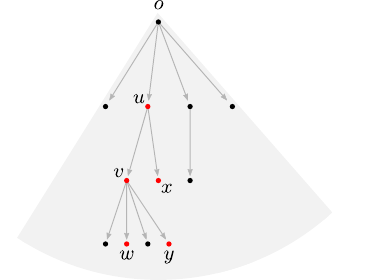}
\caption{A possible configuration for $A_t$ when $G$ is a Galton-Watson tree. 
    Vertices of $A_t$ are represented by dots: red if they belong to $R_t$ and black otherwise. 
    Here $H_t=3$, given by the red paths $(u, v, w)$ and $(u, v, y)$, whereas $M_t=5$, given by the red cluster of 
    vertices $\{u, v, w, x, y\}$.}
    \label{fig:model-scheme}
\end{figure}

\subsection{Related processes}

One of the motivations to study FPP with recovery is its connection with other epidemic and competition processes, such as the ones outlined below, including First Passage Percolation in Hostile Environment and the SIR process.

The so-called \emph{First Passage Percolation in Hostile Environment} (FPPHE) was introduced by Sidoravicius and Stauffer \cite{VA};
%
we briefly outline its definition and refer the interested reader to the references mentioned below.
On a graph $G$ place a \emph{black} particle at a reference vertex $o$, and on each vertex $x\neq o$ independently place a \emph{red particle} (called ``seed" in that context) with probability $\mu$, for some fixed $\mu \in (0,1)$.
%
The black particle initiates a FPP with rate $1$, while seeds remain inactive for as long as their hosting vertex has not been reached by the process.
When this occurs, the seed is \emph{activated}, turning red and starting a \emph{red} first-passage percolation with some rate $\lambda>0$.
FPPHE can be interpreted as a model for the spread of a disease (black), inhibited by the presence of a cure (red).
FPP with recovery can be thought of a natural modification of FPPHE. 
More precisely, FPP corresponds to setting $\lambda=1$ and $\mu=1$, as in our process.
(It would be interesting to see how different values of $\lambda$ and $\mu$ might affect the asymptotic behavior of the largest size of a red cluster in the modified version, but this is left for further investigation.)
However, in FPPHE vertices do not recover: whenever a vertex turns black (or red), then it will maintain its color forever.

For a thorough analysis of FPPHE on $\Z^d$ (for $d\geq 2$) we refer the interested reader to \cite{VA, Finn-PhD, finn2022nonequilibrium-coexistence}, whereas for questions related to hyperbolic and non-amenable graphs we point out \cite{candellero2021coexistence, candellero2021first}.

Another model that is related to the present setting was introduced in \cite{finn2022coexistence-conversion}, and consists of a first passage percolation (FPP) process with conversion rates.
In particular, a FPP process starts at a reference vertex spreading at rate 1 through vacant vertices. 
(Say that this process is of type 1.)
Each occupied vertex, independently, switches at rate $\rho>0$ to a second type of FPP (type 2) that spreads at rate $\lambda>0$ through vacant sites as well as sites occupied by type 1.
In this model, type $j\in \{1,2\}$ survives when there are vertices of type $j$ at all times.
\cite{finn2022coexistence-conversion} shows that when the underlying graph is a regular tree then coexistence is possible, while if the graph is a lattice then type 1 always dies if $\lambda$ is larger than some (small) critical value; the conjecture is that for $\lambda$ sufficiently small, type 1 survives.
One of the main differences between such a model and FPP with recovery is the fact in the former, the type-2 process can occupy vertices already reached by type 1, whereas in the present setting red FPP can only spread through vertices that have not yet been reached by the process.

Another known model related to the present framework is the so-called \emph{contact process} (cf.\ for example \cite{Liggett-StochasticInteractingSystems}
for a formal description).
In the contact process, a vertex hosting a particle is said to be \emph{infected} (here this would correspond to a \emph{red} vertex).
Infected vertices can pass on the infection to their neighbors at some constant rate $\lambda>0$, each infected vertex recovers at rate $1$ and it can become infected again. 
FPP with recovery is, however, very different from the contact process, as a vertex goes from red (infected) to black (recovered), and it will stay so forever.

The latter feature reminds of another related model, the so-called SIR model, where vertices on the graph are ``Susceptible'' or ``Infected'' or ``Recovered''. 
In the SIR model each infected vertex spreads the infection to its susceptible neighbors at a constant rate.
When a vertex recovers, it becomes immune to infection and it will be unable to spread the infection any further.
We refer to \cite{Hethcote-InfectiousDeseases} for a review on SIR.
While \cite{Hethcote-InfectiousDeseases} mainly focuses on the spread of SIR in finite graphs of large size $N$, there are also results in the case of infinite trees, for example see \cite{Besse2021SpreadingPF}.
The classical SIR model starts with a positive density of infected vertices, whereas in the present context, the infection starts from a \emph{unique} vertex (the origin) and all other vertices are seen as ``susceptible''.
Furthermore, FPP with recovery is so that a vertex can be infected by the red process even if the vertex trying to pass the infection has already recovered.
This 
represents an important difference between the two models.

\section{Analysis on the semi-line: a proof of Theorem \ref{thm_critical}}\label{sect:line}

In this section the graph $G$ is the infinite semi-line, i.e., $V=\N^\ast$;
%
we set the origin $o\coloneqq \{1\}$ and the set of edges $E\coloneqq \{(n, n+1),~n\geq1\}$. 
To simplify the notation, for all $n\geq 2$ set 
\[
T_n\coloneqq T_{(n-1, n)}.
\]
Since a cluster in $G$ is a path, then for all $t\geq0$ we have $M_t=H_t$. 
The recovery rate $\gamma>0$ is arbitrary but fixed. 
Recall that for any set $X$ we denote its cardinality by $\# X$.

\subsection{Analysis of the limsup}\label{subsect:line-limsup}

The goal of this subsection is to find a non-decreasing function $h:\R_{\geq 0}\rightarrow\R_{>0}$ such that
\[
\P \left ( \limsup_{t\rightarrow+\infty}\frac{H_t}{h(t)}=1\right ) =1.
\]
We start with a lemma that will be useful several times.

\begin{Lemma}
\label{lemma_fpp}
We have that
\[
\P \left ( \lim_{n\rightarrow+\infty} \frac{\tau_n}{n}= 1\right )=1 \quad \text{ and } \quad 
\P \left ( \lim_{t\rightarrow+\infty} \frac{\#A_t}{t}= 1 \right )=1.
\]
\end{Lemma}

\begin{proof}
By definition of $\tau_n$ (cf.\ \eqref{eq:def-reachingtime}) and since $G$ is the semi-line, then
$
\tau_n=\sum_{k=2}^n T_k
$.
Therefore, by the (strong) law of large numbers, we have that $\frac{\tau_n}{n}\to 1$ almost surely.

Next, observe that for all $t\geq\tau_n$ we have $\#A_t\geq n$.
Therefore, when $t\to \infty$ one has $\#A_t\rightarrow \infty$ almost surely (since for all $n$, $\tau_n<\infty$ almost surely).
Thus, $\frac{\tau_{\#A_t}}{\#A_t} \to 1 $ and $\frac{\tau_{\#A_t+1}}{\tau_{\#A_t}}\to 1$ almost surely.
Since for all $t\geq 0$ one has $\tau_{\#A_t}\leq t < \tau_{\#A_t+1}$, then, $\frac{\#A_t}{t}\to 1$ as claimed.
\end{proof}

Next, we introduce an auxiliary object, which we call the \emph{tail red cluster}.
Recall that $\range{a}{b}$ denotes the connected path going from vertex $a$ to vertex $b$ in $G$.

\begin{Definition} 
\label{def_tail_clusters}
For $n\geq 1$ let $\mathbf{T}_n$ denote the largest red cluster present at time $\tau_n$ \emph{which includes vertex} $n$; we call $\mathbf{T}_n$ the \emph{tail red cluster} at time $\tau_n$.
Moreover, for any $n\geq 1$ let $\widehat{H}_n$ be its size, i.e.,
\[
\widehat{H}_n\coloneqq \# \mathbf{T}_n = \max\{m\in\range{1}{n}:\range{n-m+1}{n}\subset R_{\tau_n}\}.
\]
\end{Definition}

The following result shows that, typically, the largest red cluster is the tail red cluster.

\begin{Proposition}
\label{prop_tail_cluster}
Let $h:\N^*\rightarrow\R_{>0}$ be a non-decreasing function, then:
\begin{equation}\label{eq:equality-tail-cluster}
\limsup_{n\rightarrow+\infty} \frac{\widehat{H}_n}{h(n)}=\limsup_{t\rightarrow +\infty} \frac{H_t}{h(\#A_t)}, \quad\P\text{-almost surely.}
\end{equation}
\end{Proposition}

\begin{proof}
We start by showing that,  almost surely,
$
\limsup_{n} \frac{\widehat{H}_n}{h(n)}\leq \limsup_{t} \frac{H_t}{h(\#A_t)}.
$

By construction, the size of the tail red cluster is at most the size of the largest red cluster, that is, $\widehat{H}_n\leq H_{\tau_n}$, and by definition $\#A_{\tau_n}=n$. 
Moreover $\tau_n\rightarrow \infty$ by Lemma \ref{lemma_fpp}, hence
\[
\limsup_{n\to+\infty} \frac{\widehat{H}_n}{h(n)} \leq \limsup_{n\to+\infty}\frac{H_{\tau_n}}{h(\#A_{\tau_n})}\leq \limsup_{t\to+\infty} \frac{H_t}{h(\#A_t)}.
\]
%
Now we proceed with the reverse inequality.
Recall that $\mathbf{C}_i$ denotes the recovery time of vertex $i$.
W.l.o.g.\ assume $\limsup_{t} \frac{H_t}{h(\#A_t)}>0$, (otherwise the result is clear), and fix $\ell \in \left (0, \, \limsup_{t} \frac{H_t}{h(\#A_t)} \right )$.
Fix an arbitrary value $n_0\geq 1$ and define
\[
t_0\coloneqq \max_{i\in\range{1}{n_0}}(\tau_i+\mathbf{C}_i),
\]
that is the time from which the first $n_0$ vertices are black.
Almost surely $t_0<+\infty$ and by definition of $\limsup$, there is a time $\hat{t}\geq t_0$ so that 
\begin{equation}\label{eq:Hhat}
\frac{H_{\hat{t}}}{h(\#A_{\hat{t}})}\geq \ell.
\end{equation}
Now let $N\in\range{1}{\#A_{\hat{t}}}$ denote the rightmost vertex of a red cluster with size $H_{\hat{t}}$.
If the cluster is not unique, just pick one arbitrarily. 
Clearly, since $\hat{t}\geq t_0$, then $N\geq n_0$ by definition of $t_0$.

We now claim that $\widehat{H}_N\geq H_{\hat{t}}$ and $\hat{t}\geq\tau_N$ (see Figure \ref{fig:proof_back_in_time}). 
To see this, we reason as follows.
Since $N\leq \# A_{\hat{t}}$, then by time $\hat{t}$ vertex $N$ must be already part of the aggregate, implying $\hat{t}\geq\tau_N$.
Moreover, between times $\tau_N$ and $\hat{t}$ there might have been some recoveries in the cluster of size $\widehat{H}_N$, implying  $\widehat{H}_N\geq H_{\hat{t}}$.
Using the facts that $N\geq n_0$ and that both functions $h$ and $t\mapsto \#A_t$ are non-decreasing, we obtain
\[
\widehat{H}_N  \stackrel{\eqref{eq:Hhat}}{\geq} \ell h(\#A_{\hat{t}}) \ \stackrel{\hat{t}\geq\tau_N }{ \geq} \ \ell h(\#A_{\tau_N})=\ell h(N).
\]
By putting everything together, we have that
\[
\limsup_{n\to+\infty} \frac{\widehat{H}_n}{h(n)}\geq \ell, \quad\P\text{-almost surely}.
\]
By letting $\ell \rightarrow \limsup_{t} \frac{H_t}{h(\#A_t)}$ we obtain the claim.
\end{proof}

\begin{figure}[h!]
    \centering
    \includegraphics[height=22mm]{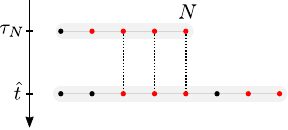}
    \caption{A representation of the reasoning behind the proof of Proposition \ref{prop_tail_cluster}; possible configuration for $A_{\hat{t}}$ and $A_{\tau_N}$ where $N$ is the rightmost vertex of a maximizing cluster for $H_{\hat{t}}$.}
    \label{fig:proof_back_in_time}
\end{figure}

\begin{Remark}
On the right-hand side of \eqref{eq:equality-tail-cluster}, the denominator is not a deterministic function of $t$. 
However, we will be able to get rid of the randomness with the help of Lemma \ref{lemma_fpp}. 
\end{Remark}

The next step is to find a non-decreasing function $h:\N^*\rightarrow\R_{>0}$ such that 
\begin{equation}\label{eq:equation-h}
\limsup_{n\rightarrow+\infty}\frac{\widehat{H}_n}{h(n)}=1,\quad\P\text{-almost surely}.
\end{equation}
In order to do this, we will compute the distribution of $\widehat{H}_n$ for all $n\geq1$.

\begin{Lemma} 
\label{lemma_tail_distrib}
Recall that $\gamma>0$ is the recovery rate of the red vertices.
For all $n\geq 1$ and $m\in\range{0}{n}$,
\[
\Proba(\widehat{H}_n\geq m)=\left[\prod_{k=1}^{m-1} (1 + k\gamma)\right]^{-1}.
\]
\end{Lemma}

\begin{proof}
Observe that 
\[
\{\widehat{H}_n\geq m\}=\bigcap_{j=n-m+1}^n\left\lbrace \mathbf{C}_j>\sum_{k=j+1}^n T_k\right\rbrace .
\]
Now we exploit independence of the recovery clocks $(\mathbf{C}_k)_{k\geq1}$ as well as independence between $(\mathbf{C}_k)_{k\geq1}$ and the passage times.
This gives 
\[
\begin{split}
\Proba(\widehat{H}_n\geq m) & =\E \left [\Proba(\widehat{H}_n\geq m~|~(T_k)_{k\geq1}) \right ] = \E\left[\prod_{j=n-m+1}^n\exp\left(-\gamma\sum_{k=j+1}^m T_k\right)\right]\\
 &=\E\left[\exp\left(-\gamma\sum_{k=1}^{m-1} kT_{n-m+1+k}\right)\right] = \prod_{k=1}^{m-1}\int_0^{+\infty}e^{-k\gamma x}e^{-x}\mathrm{d}x = \left[\prod_{k=1}^{m-1} (1+k\gamma)\right]^{-1},
\end{split}
\]
as claimed.
\end{proof}

It follows immediately from Lemma \ref{lemma_tail_distrib} that the sequence of random variables $(\widehat H_n)_{n\in\N^*}$ converges in distribution.

To simplify the notation, for all $m\geq 0$ set
\[
\Pi(m)\coloneqq \left[\prod_{k=1}^{m-1} (1 + k\gamma)\right]^{-1}.
\]
%
%
In the following, $\Gamma$ denotes the usual \emph{Gamma function}.
By its fundamental properties, for all $m\geq 1$ we have 
\[
\Pi(m)=\gamma^{-(m-1)}\frac{\Gamma(\gamma^{-1})}{\Gamma(m+\gamma^{-1})}.
\]
Using Stirling's approximation it follows that
\begin{equation}
    \label{eq_Pi_asympt}
    \log\Pi(m)\sim-m\log m\quad\text{as $m\rightarrow+\infty$}.
\end{equation}
We now show that $x\mapsto\frac{\log x}{\log\log x}$ is a good candidate for the function $h$ in \eqref{eq:equation-h}.

\begin{Theorem}
\label{thm_limsup_tail}
Almost surely, $\limsup_{n\rightarrow+\infty} \frac{\widehat{H}_n\log\log n}{\log n}=1$.
\end{Theorem}

\begin{proof}
We start by showing that for all $r>1$ we have $\Proba \left [\limsup_{n} \frac{\widehat{H}_n\log\log n}{\log n}>r \right ]=0$.
Fix any $r>1$, and for all $n\geq \lceil e^e\rceil$ define 
\[
m_n\coloneqq \left \lceil r\frac{\log n}{\log\log n}\right \rceil.
\] 
By Lemma \ref{lemma_tail_distrib} we have
\begin{equation}
    \label{eq_bc_series}
    \sum_{n\geq\lceil e^e\rceil} \Proba\left(\widehat{H}_n\geq r\frac{\log n}{\log\log n}\right)
    =\sum_{n\geq \lceil e^e\rceil}\Pi(m_n)\mathbf{1}_{m_n\leq n}\leq\sum_{n\geq \lceil e^e\rceil}\Pi(m_n).
\end{equation}
Now, observe that by (\ref{eq_Pi_asympt}), as $n\rightarrow+\infty$ we have
\[
\log\Pi(m_n)\sim-m_n\log m_n\sim-r\log n.
\]
As $r>1$, this implies that the series in (\ref{eq_bc_series}) converges. 
Then by the first Borel-Cantelli lemma, almost surely only finitely many events of the sequence $(\{\widehat{H}_n\geq r\frac{\log n}{\log\log n}\})_{n\geq\lceil e^e\rceil}$ can be realized.
Since $r>1$ was arbitrary,
\[
\P \left ( \limsup_{n\to+\infty} \frac{\widehat{H}_n\log\log n}{\log n}\leq 1 \right ) =1.
\] 
Now we show that $\Proba \left [\limsup_{n} \frac{\widehat{H}_n\log\log n}{\log n}\geq r \right ]=1$, for all $r<1$.
Let $r\in(0,1)$ and fix $s\in(1,\frac{1}{r})$. 
Similarly to what we did in the first part, for all $n\geq \lceil e^e\rceil+1$ define $k_n\coloneqq \lfloor n^s\rfloor$ and $w_n\coloneqq \left \lceil r\frac{\log k_n}{\log\log k_n}\right \rceil$. By definition,
\[
\frac{w_{n+1}}{k_{n+1}-k_n}\sim\frac{r\log n}{n^{s-1}\log\log n}\rightarrow0\quad\text{as}\quad n\rightarrow+\infty.
\]
For all large enough $n$, we also have $w_n\leq k_n$, hence there is $n_0\geq\lceil e^e\rceil+1$ so that for all $n\geq n_0$ we have 
\[
w_n\leq k_n \quad \text{ and }\quad k_{n+1}-w_{n+1}>k_n. 
\]
By construction this implies that the events $(\{\widehat{H}_{k_n}\geq w_n\})_{n\geq n_0}$ are independent. Analogously to what done above, by lemma \ref{lemma_tail_distrib},
\begin{equation}
    \label{eq_bc_series2}
    \sum_{n\geq n_0}\Proba\left(\widehat{H}_{k_n}\geq r\frac{\log k_n}{\log\log k_n}\right)=\sum_{n\geq n_0} \Pi(w_n)\mathbf{1}_{w_n\leq k_n}=\sum_{n\geq n_0} \Pi(w_n).
\end{equation}
Finally, by reasoning as above, as $n\rightarrow+\infty$,
\[
\log\Pi(w_n)\sim-w_n\log w_n\sim-rs\log n.
\]
Since $rs<1$, the series in \eqref{eq_bc_series2} diverges.
By the second Borel-Cantelli lemma, almost surely there is an infinite sub-sequence of events in $(\{\widehat{H}_{k_n}\geq w_n\})_{n\geq n_0}$ that will be realized, giving that
$\Proba \left [\limsup_{n} \frac{\widehat{H}_n\log\log n}{\log n}\geq r \right ]=1$. 
Since $r\in(0,1)$ was arbitrary, the claim follows.
\end{proof}

Finally we are ready to prove part (i) of Theorem \ref{thm_critical}.

\begin{proof}[Proof of Theorem \ref{thm_critical}(i)]
It suffices to show that, almost surely, 
\[
\limsup_{t\rightarrow+\infty}\frac{H_t\log\log t}{\log t}=\limsup_{t\rightarrow+\infty}\frac{H_t\log\log(\#A_t)}{\log(\#A_t)}=1.
\]
The second equality is given by Theorem \ref{thm_limsup_tail} and Proposition \ref{prop_tail_cluster} since $(e^e,+\infty)\ni x\rightarrow \frac{\log x}{\log\log x}$ is increasing. 
Finally, since $\#A_t\sim t$ almost surely by Lemma \ref{lemma_fpp}, the result follows.
\end{proof}

\subsection{Analysis of the liminf}\label{subsect:line-liminf}

One way to show the sought result is the following.
(We are grateful to an anonymous referee for suggesting this short proof.)


First, define a sequence of stopping times $\{\sigma_n\}_{n\in\N}$ as follows:
\[
\sigma_0\coloneqq  0 \quad \text{and for }n\geq 0, \quad \sigma_{n+1}\coloneqq \inf \{t> \sigma_n \ : \ \#R_t\neq\#R_{\sigma_n}\}.
\]
This sequence corresponds to the jumping times of the process $(\#R_t)_{t\geq 0}$. For simplicity of notation, for all $n\geq 0$ let
\[
W_n\coloneqq \#R_{\sigma_n}.
\]
Using the fact that the exponential distribution is memoryless, we have that the process $\{W_n\}_{n\in\N}$ is a Markov chain, with
\[
\P(W_{n+1}-W_n=1~|~W_n)=1-\P(W_{n+1}-W_n=-1~|~W_n)=\frac{1}{1+\gamma W_n} \quad \text{for all $n\in\N^*$}.
\]
Now, observe that $\P(W_{n+1}-W_n=1~|~W_n)\leq 1/2$ iff $W_n\geq \left\lceil \gamma^{-1}\right \rceil$. Therefore, if we see this process as a random walk on $\N\setminus \range{0}{\left\lceil \gamma^{-1}\right \rceil}$ reflected at $\left\lceil \gamma^{-1}\right \rceil $, we see that its transition probabilities are dominated by those of the simple random walk (on $\N\setminus \range{0}{\left\lceil \gamma^{-1}\right \rceil}$ reflected at $\left\lceil \gamma^{-1}\right \rceil $), which is recurrent.

This argument shows that the Markov chain $(W_n)_{n\in\N^*}$ visits vertex $\left\lceil \gamma^{-1}\right \rceil $ infinitely often almost surely. Since $\P\left (W_{n+\left\lceil \gamma^{-1}\right \rceil}=0 \mid W_n=\left\lceil \gamma^{-1}\right \rceil \right )$ is positive and does not depend on $n\in\N$, we deduce that $(W_n)_{n\in\N^*}$ visits $0$ infinitely often almost surely, that is
\begin{equation}
    \label{eq_liminf_Rsigman}
    \liminf_{n\rightarrow+\infty}\#R_{\sigma_n}=0
\end{equation}
Now, fix $n\in\N^*$ and observe that at time $\tau_n$ the process $(\#R_t)_{t\geq 0}$ may have had at most $2n-2$ jumps: exactly $n-1$ positive jumps corresponding to new vertices being reached, and at most $n-1$ negatives jumps corresponding to recoveries; giving $\sigma_{2n-2}\geq\tau_n$. Since $\tau_n\rightarrow+\infty$ almost surely as $n\rightarrow+\infty$ by Lemma \ref{lemma_fpp}, we deduce that $\sigma_n\rightarrow+\infty$ almost surely as as $n\rightarrow+\infty$. Together with (\ref{eq_liminf_Rsigman}) this implies that almost surely there are arbitrary large time $t$ so that there is no red vertices at time $t$, giving the sought liminf.


\begin{Remark}
It is possible to show the result on the liminf also in a different way, which investigates the process directly.
Since we  believe it to be of independent interest, we present it in Appendix \ref{sect:appendix}.
\end{Remark}

\section{Supercritical Galton-Watson trees}\label{sect:GWtree}

In this section $G$ is a supercritical Galton-Watson tree and the recovery rate $\gamma>0$ is arbitrary. 
More precisely, let $\zeta$ denote a distribution on the non-negative integers and consider a random variable $B\sim \zeta$. 
For all $k\geq 0$, let $\zeta(k)\coloneqq \P(B=k)$.
We require that $\E B\coloneqq \sum_{k\geq 1}k\zeta (k) >1$. 
Then, we assume that $\zeta$ is the offspring distribution that characterizes $G$.
We let $\mathcal S$ denote the event that $G$ is infinite.
Furthermore, define
\[
\Proba_\mathcal S (\cdot)\coloneqq \Proba ( \cdot \mid \mathcal S). 
\]
To simplify the notation, for all $v\in  V\setminus\{o\}$ let $v^-$ denote the predecessor of $v$ in the (unique) non-backtracking path from the root $o$ to $v$ and set $T_{v}\coloneqq T_{(v^-, v)}$.
We also define 
\begin{equation}\label{eq:def-alpha}
\boldsymbol{\alpha}\coloneqq \E B-1.
\end{equation}
See Figure \ref{fig:simulation} for a simulation of FPPHE with recovery on $G$.

\begin{figure}[ht]
    \centering
    \includegraphics[height=5cm]{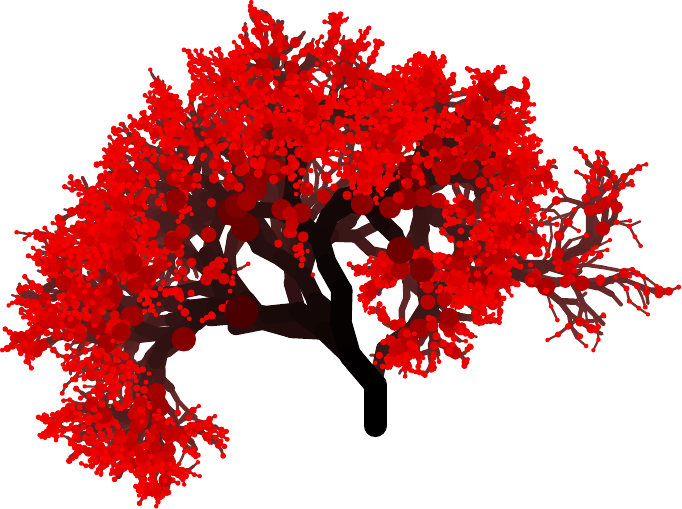}
    \caption{Simulation of $A_t$ with $B\sim\mathrm{Bin}(2, \frac{4}{5})$, $\gamma=\frac{1}{2}$ and $t=13$. 
    Red dots are vertices of $R_t$ and dark lines are edges inside $A_t$.}
    \label{fig:simulation}
\end{figure}

Our next aim is to find two non-decreasing functions $h, g:\R_{\geq 0}\rightarrow\R_{>0}$ so that $\P_\mathcal S$-almost surely,
\[
\liminf_{t\rightarrow+\infty}\frac{H_t}{h(t)}\geq1\quad\text{and}\quad\liminf_{t\rightarrow+\infty}\frac{M_t}{g(t)}\geq1.
\]

\subsection{Useful notation and properties}\label{subsect:GW-properties}

We start by defining two sequences of stopping times that will be widely used later on.

\begin{Definition}
\label{def-reachingtimes2}
For all $n\geq1$, set 
$\theta_n\coloneqq \inf \{t\geq0:\#A_t\geq n\}$.
Moreover, for all $v\in V$, let 
$\tau_v\coloneqq \inf \{t>0 \ : \ v\in A_t\}$.
\end{Definition}
Fix an integer $n\geq 1$ and assume $\#G\geq n$.
Then $\theta_n$ is the reaching time of some vertex, and $\#A_{\theta_n}=n$.
For any subset $A \subset V$ we define $\partial^* A$ to be its external boundary, i.e.,
\[
\partial^* A \coloneqq  \{w \in V\setminus A \ : \ \{v,w\}\in E , \text{ for some }v\in A\}.
\]
In this subsection we provide asymptotic estimates for $\theta_n$ and $\#\partial^*A_{\theta_n}$ as $n \to \infty$. 
The following lemma gives a characterization for the distribution of $(\#\partial^*A_{\theta_n})_{n\geq 1}$. 
Recall that the offspring distribution $\zeta$ is described at the beginning of Section \ref{sect:GWtree}.

\begin{Lemma}
\label{lemma_construction}
Consider a sequence $(D_i)_{i\geq 1}$ of i.i.d.\ random variables distributed according to the offspring distribution $\zeta$ and define a new sequence of random variables $(J_n)_{n\geq 0}$ inductively by setting
$J_0\coloneqq 1$, and for all $n\geq 0$  
\[
J_{n+1}\coloneqq J_n+(D_{n+1}-1)\mathbf{1}_{\{J_n>0\}}.
\]
Then, we have that $(\#\partial^*A_{\theta_n})_{n\geq 1}$ is distributed as $(J_n)_{n\geq 1}$.
\end{Lemma}

\begin{proof}
Observe that conditionally on all the passage times, $(A_{\theta_n})_{n\geq 1}$ is a deterministic exploration process of $G$.
More precisely, we start at the origin and we consecutively discover, one at the time, vertices of the external boundary of the currently explored region. 
In particular, we have that for all $n\geq 1$,
\[
\#\partial^*A_{\theta_{n+1}}-\#\partial^*A_{\theta_n}
\]
is independent of $(\#\partial^*A_{\theta_k})_{1\leq k\leq n}$ and, on the event $\#\partial^*A_{\theta_n}>0$, has distribution $\zeta-1$. The result follows then by induction.
\end{proof}

\begin{figure}[h!]
    \centering
    \includegraphics[height=3cm]{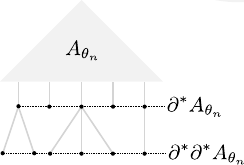}
    \caption{Illustration of the step-by step construction 
    in the proof of Lemma \ref{lemma_construction}.}
    \label{fig:step_by_step}
\end{figure}

\begin{Lemma}
\label{lemma_controlled_growth}
Recall the definition of $\boldsymbol{\alpha}$ from \eqref{eq:def-alpha}. 
Almost surely on $\mathcal S$, we have
\[
\frac{\theta_n}{\log n}\to \frac{1}{\boldsymbol{\alpha}} \quad \text{ and } \quad 
\frac{\#\partial^*A_{\theta_n}}{n}\to \boldsymbol{\alpha}.
\]
\end{Lemma}

\begin{proof}
The first statement follows directly by applying the methods in \cite{AK} (which are also deeply discussed in \cite[Sections I.12 and III.9]{AthreyaNey}).
Now we can apply Lemma \ref{lemma_construction}.
Since the event $\{G$ is infinite$\}$ can be rewritten as 
$\{\forall n\geq 1,~\#\partial^*A_{\theta_n}>0\}$, 
by Lemma \ref{lemma_construction} it suffices to show that, conditional on $\mathcal S'\coloneqq \{\forall n\geq 1,\, ~J_n>0\}$, almost surely $\frac{J_n}{n}\to \boldsymbol{\alpha} $.
In order to show this, observe that, on $\mathcal S'$, we have $J_n=1+\sum_{k=1}^n(D_k-1)$.
By the strong law of large numbers, almost surely on $\mathcal S'$, we have $\frac{J_n}{n}\to \boldsymbol{\alpha}$, giving the second statement. 
\end{proof}

Now we introduce a class of auxiliary processes coupled with the original one.

\begin{Definition}
\label{def_shifted_process}
Let $v\in V$.
We define the \emph{process shifted by $v$} to be the translated process started at a new root, namely $v$.
In particular, say that now $v\neq o$ is the root, then we \emph{discard} the edge connecting $v$ to $v^-$ and consider only the connected component containing $v$.
This guarantees that the distribution of the process does not change.
All quantities in the new version will present a superscript $(v)$, however, the random variables associated to each edge (passage times) and to each vertex (recovery clocks) will not be changed.
(E.g., $C_v^{(v)} \equiv C_v$.)
\end{Definition}
This procedure is well defined because the Galton-Watson tree is invariant (in a distributional sense) with respect to translation.
Moreover, the process started at any vertex $v$ at time $\tau_v$ has the same distribution as the original one.
In any collection of vertices so that none of them is a prefix of any other, the corresponding shifted processes are all independent of one another and distributed as the original one.

\begin{Proposition}
\label{prop_split_property}
There is a random set of vertices $U\subset V$ independent of the passage times and recovery clocks such that,
\begin{itemize}
    \item[(i)] $\Proba_\mathcal S$-almost surely we have 
    $1<\#U<+\infty$ and $\#\left (V\setminus\bigcup_{v\in U}V^{(v)}\right ) < \infty$.
    \item[(ii)] Under $\Proba_\mathcal S$ and conditional on $U$, $(G^{(v)})_{v\in U}$ is a collection of independent random graphs, with the same distribution as $G$. \end{itemize}
\end{Proposition}

\begin{proof}
Proof of (i).
For any $v\in V$, consider the set of ``sons'' (i.e., direct descendants) of $v$ which are roots of infinite sub-trees in $G$, and denote it by
$
\mathbf S_v\coloneqq \{w\in \partial^*\{v\}:\#G^{(w)}=+\infty\}
$.

First, observe that $\P_{\mathcal S}$-almost surely there exists a vertex $v\in G$ such that $\#\mathbf S_v>1$. 
(Under $\P_{\mathcal S}$, $\{v\in G:\mathbf S_v\geq 1\}$ is an almost surely infinite supercritical Galton-Watson tree.)
Then, note that if $u,v \in V$ have a common ancestor $a\notin\{u, v\}$ and they are such that $\#\mathbf S_u>1$ and $\#\mathbf S_v\geq1$, then $\#\mathbf S_a>1$. 
It follows that there is a 
(random) vertex $\mathbf{w}\in V$ 
that, $\P_{\mathcal S}$-almost surely, is the unique vertex such that $\#\mathbf S_\mathbf{w}>1$ and $\#\left ( V\setminus V^{(\mathbf{w})}\right ) < \infty$. 
We set $U\coloneqq \mathbf S_{\mathbf{w}}$; see Figure \ref{fig:split_prop}.

Proof of (ii).
The second item follows directly from the definition of $U$, since $\mathbf S_{\mathbf{w}}$ is a (finite) set of ``siblings'' (hence not descendants of one another), and from the distributional invariance of the Galton-Watson tree.
\end{proof}

\begin{figure}[h!]
    \centering
    \includegraphics[height=4.5cm]{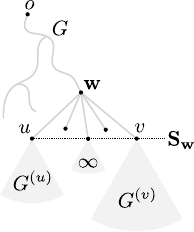}
    \caption{An illustration of the proof of Proposition \ref{prop_split_property}.}
    \label{fig:split_prop}
\end{figure}

\subsection{Analysis of the liminf}\label{subsect:GW-general}

In this subsection we shall prove results that will be applied to the sequences $(H_t)_{t\in\R}$ and $ (M_t)_{t\in\R}$, hence write $(\mathcal Q_t)_{t\in\R}$ 
to denote either of them.
In the next subsection we shall see how to use this in order to prove Theorem \ref{thm_supercritical}.
In a natural way, we set $\mathcal{Q}_t\coloneqq 0$, for all $t<0$.

The first goal is to provide a sufficient property to characterize the asymptotic lower bound for $\mathcal Q_t$, as $t\to \infty$.
More precisely, if $f:\R_{\geq 0}\rightarrow\R_{>0}$ is a non-decreasing function we want to get a non-trivial property on $f$ that guarantees 
$\liminf_{t}\frac{\mathcal Q_t}{f(t)}\geq 1$, $\P_{\mathcal{S}}$-almost surely.

In order to do so, we start by looking for an upper bound on $\P \left (\mathcal Q_t\leq m \right )$ when $(t, m)\in\R_{\geq 0}^2$ are large. 
We shall later apply the first Borel-Cantelli lemma to show that when $m=m(t)$ has the claimed expression, then the event $\{\mathcal Q_t\leq m\}$ occurs only finitely many times almost surely.
Below we make use of the fact that a Galton-Watson tree is self-similar (in a distributional sense) and, when it survives, it grows exponentially fast. 
For all $(t, m, n)\in\R_{\geq 0} \times \N \times \N$ let

\begin{equation}\label{eq:def-Pmnt}
P_m^n(t)\coloneqq \Proba[\mathcal Q_t\leq m, ~\#\partial^* A_{t-1}\geq n].
\end{equation}
Let $T \sim \mathrm{Exp}(1)$ be a random variable independent of the whole process.
For all $m\in\N^*$ set
\begin{equation}\label{eq:def-etam}
\eta(m)\coloneqq \Proba[\mathcal Q_{1-T}\leq m].
\end{equation}

\begin{Proposition}
\label{prop_boundary_ineq}
For all $(t, m, n)\in\R_{\geq 0} \times \N \times \N$, 
we have 
\[
P_m^n(t) \leq \eta(m)^n,
\]
where $P_m^n(t)$ and $\eta(m)$ are defined in \eqref{eq:def-Pmnt} and \eqref{eq:def-etam} respectively.
\end{Proposition}

\begin{proof}
By construction, for all fixed $t \geq 0$, the set $\partial^*A_{t-1}$ is almost surely finite.
For all $n\geq 1$ let $\mathcal B_{n+}$ denote the set all possible shapes of size 
at least $n$ for $\partial^*A_{t-1}$, i.e. the set of all subsets of vertices $F$ such that $\#F\geq n$ and $ \Proba(\partial^*A_{t-1}=F)>0$.
Hence, by definition of $P_m^n(t)$ and of the shifted process, one has
\[
\begin{split}
P_m^n(t)
& =\sum_{F\in\mathcal B_{n+}} \Proba(\{ \mathcal Q_t\leq m\}\cap \{\partial^*A_{t-1}=F\}) \\
& \leq\sum_{F\in\mathcal B_{n+}}\Proba[\{ \forall v\in F, \,\mathcal Q_{t-\tau_v}^{(v)}
\leq m\}\cap \{ \partial^*A_{t-1} = F \} ],
\end{split}
\]
where $\tau_v$ was introduced in Definition \ref{def-reachingtimes2}. 
Fix any $F\in\mathcal B_{n+}$. 
By the memoryless property of the exponential distribution, when we condition on $\{\partial^*A_{t-1}=F\}$, the set $(\tau_v-t+1)_{v\in F}$ is a collection of i.i.d.\ $\mathrm{Exp}(1)$ random variables. 

Now consider the collection of processes shifted by $v$, where $v$ varies in $F$.
Thus, conditional on $\{\partial^*A_{t-1}=F\}$, such collection consists of
i.i.d.\ processes, distributed as the original one,  independent of $(\tau_v)_{v\in F}$. 
Therefore, 
\[
\Proba[\forall v\in F, \,\mathcal Q_{t-\tau_v}^{(v)}\leq m \mid F]
= \left ( \Proba[\mathcal Q_{1-T} \leq m]\right )^{\#F}  \leq \eta(m)^n
\]
(where the last inequality follows from the fact that $\#F\geq n$),
and hence
\[
\begin{split}
P_m^n(t) & \leq\sum_{F\in\mathcal B_{n+}}  \eta(m)^n
\Proba(\partial^*A_{t-1}=F) 
= \eta(m)^n \Proba(\#\partial^*A_{t-1}\geq n).
\end{split}
\]
Hence the claim follows.
\end{proof}
An illustration of the reasoning in the above proof is given in Figure \ref{fig:boundary_ineq}.

\begin{figure}[h!]
    \centering
    \includegraphics[height=3.8cm]{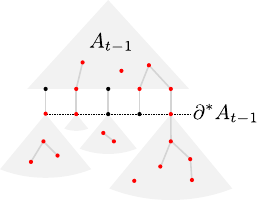}
    \caption{
    Possible configuration of the red vertices at time $t$.}
    \label{fig:boundary_ineq}
\end{figure}

The next auxiliary lemma gives a bound on the desired liminf restricted to a random sequence of times.
Recall the definition of $\boldsymbol{\alpha}$ from \eqref{eq:def-alpha}.

\begin{Lemma}
\label{lemma_sub_seq}
Let $(Y_n)_{n\geq1}$ and $(Z_n)_{n\geq1}$ be two random sequences 
so that $(Z_n)_{n\geq1}$ is independent of the process of FPP with recovery, such that $\P (\frac{Y_n}{\log n}\to \boldsymbol{\alpha}^{-1})=1 $ and $\P (\frac{Z_n}{\log n}\to \boldsymbol{\alpha}^{-1})=1 $.
Let $\eta (\cdot)$ be as in Proposition \ref{prop_boundary_ineq} and let $h:\R_{\geq 0}\rightarrow\R_{>0}$ be a non-decreasing function so that, for all $x>0$ large enough,
\begin{equation}\label{eq:assump-etah}
\eta \left ( h(x)\right )\leq \exp(-1/x).
\end{equation}
Then, for all $\beta\in(0,\boldsymbol{\alpha})$,
\[
\liminf_{n\rightarrow +\infty}\frac{\mathcal Q_{Z_n}}{h(e^{\beta Y_n})}\geq 1,\quad \P_{\mathcal{S}}\text{-almost surely}.
\]
\end{Lemma}

\begin{proof}
Fix $y\in(\frac{\beta}{\boldsymbol{\alpha}},1)$ and $z\in(y,1)$. 
For all $n\geq 1$ define
\[
B_n\coloneqq \{\mathcal Q_{Z_n}\leq h(e^{\beta Y_n})\} \quad \text{ and }\quad G_n\coloneqq \{\#\partial^*A_{Z_n-1}\geq \beta n^z, ~h(e^{\beta Y_n})\leq h(n^y)\}.
\]
Since $(Z_n)_{n\geq1}$ is independent of the process, then for all $n\geq1$, 
\[
\Proba(B_n\cap G_n)\leq\Proba[\mathcal Q_{Z_n}\leq h(n^{y}), \ \#\partial^*A_{Z_n-1}\geq \beta n^z]=\E\left [P_{h(n^y)}^{\beta n^z}(Z_n) \right ],
\]
where the last equality follows from the definition in \eqref{eq:def-Pmnt}.
Thus, by Proposition \ref{prop_boundary_ineq} and by assumption \eqref{eq:assump-etah}, 
%
%
we get that for all $n\geq 1$ sufficiently large
\[
\begin{split}
\Proba(B_n\cap G_n) & \leq \E \left [P_{h(n^y)}^{\beta n^z}(Z_n) \right ] \ \stackrel{\text{Prop.\ \ref{prop_boundary_ineq}}}{\leq}  \E \left [ \eta \left ( h(n^y) \right )^{\beta n^z}\right ]
\stackrel{\eqref{eq:assump-etah} }{\leq }\exp \left ( -\frac{1}{n^y}\, \beta n^z \right ) = \exp(-\beta n^{z-y}).
\end{split}
\]
Since $z-y>0$, it follows that $\sum_{n\geq 1}\Proba(B_n\cap G_n)<\infty$. 
Moreover, for all event $E$ so that $\Proba (E)=0$ we also have $\Proba_\mathcal S (E)=0$.
Thus, by the first Borel-Cantelli lemma
\begin{equation}
    \label{eq_limsup_bc}
    \Proba_\mathcal S\left[\limsup_{n\to+\infty}(B_n\cap G_n)\right]=0.
\end{equation}
Our next aim is to show that the sequence of events $(G_n)_{n\geq1}$ does not give any real contribution in the above intersection, implying that \eqref{eq_limsup_bc} only depends on the limsup of the events $(B_n)_{n\geq1}$.


Since $y/\beta>1/\boldsymbol{\alpha}$, from the asymptotic behavior of $(Y_n)_{n\geq1}$ it follows that almost surely for all $n\geq 1$ large enough we have $Y_n\leq\frac{y\log n}{\beta}$. Thus, since $h$ and the exponential function are non-decreasing, $\P_\mathcal S$-almost surely,
\begin{equation}\label{eq:G1}
h(e^{\beta Y_n})\leq h(n^{y}) \quad \text{for all $n$ sufficiently large.}
\end{equation}
Using the assumption on the asymptotic behavior of $(Z_n)_{n\geq1}$ and the fact that $z< 1$, together with Lemma \ref{lemma_controlled_growth}, it follows that $\P_\mathcal S$-a.s.\ for all $n$ and $k$ large enough,
\[
Z_n-1 \geq \frac{\sqrt{z}}{\boldsymbol{\alpha}} \log n \quad \text{ and } \quad \theta_k\leq \frac{1}{\sqrt{z} \boldsymbol{\alpha}}\log k.
\]
Thus, for all $n$ and $k$ large enough, $ Z_n-1-\theta_k\geq\frac{\sqrt{z}}{\boldsymbol{\alpha}}\log n-\frac{1}{\sqrt{z} \boldsymbol{\alpha}}\log k $.
Setting $k=\lfloor n^z\rfloor$, we get that $\P_{\mathcal S}$-almost surely,
\begin{equation}\label{eq:zn-1}
Z_n-1\geq\theta_{\lfloor n^z\rfloor} \quad \text{for all $n$ sufficiently large}.
\end{equation}
Using again Lemma \ref{lemma_controlled_growth}, since $\beta<\boldsymbol{\alpha}$, $\P_\mathcal S$-a.s.\ for all $n\geq1$ large enough we have 
$\#\partial^*A_{\theta_n}\geq\beta(n+1)$, which implies that 
for all $n$ large and $t\geq\theta_{\lfloor n^z\rfloor}$ we have $\#\partial^*A_t\geq\beta n^z$. 
As a consequence, 
by \eqref{eq:zn-1} one obtains that there is a $\P_{\mathcal{S}}$-almost surely finite value $n_0$ so that, for all $n\geq n_0$,
\begin{equation}\label{eq:G2}
\#\partial^*A_{Z_n-1}\geq \beta n^z.
\end{equation}
Putting together equations \eqref{eq:G1} and \eqref{eq:G2} gives that there is a $\P_{\mathcal{S}}$-almost surely finite value $n_1$ so that, for all $n\geq n_1$, $G_n$ is realized. 
Together with \eqref{eq_limsup_bc} we obtain
$
\Proba_\mathcal S\left[\limsup_{n}B_n\right]=0
$,
which implies the claim.
\end{proof}

Finally, the previous lemma and Proposition \ref{prop_split_property} will be used to obtain a sufficient characterization for an asymptotic lower bound.

\begin{Theorem}
\label{thm_lower_bound}
Let $h:\R_{\geq 0}\rightarrow\R_{>0}$ be a non-decreasing function such that $\eta \left ( h(x)\right )\leq \exp(-1/x)$ for all $x$ large enough. 
Then, for all $\beta\in(0,\boldsymbol{\alpha})$ we have
\[
\liminf_{t\rightarrow+\infty}\frac{\mathcal Q_t}{h(e^{\beta t})}\geq 1, \quad
\P_{\mathcal{S}}\text{-almost surely}.
\]
\end{Theorem}

\begin{proof}
For all $t_0\geq 0$ define
\[
\mathcal Q_{t_0}^-\coloneqq \lim_{t\to t_0^-} \mathcal Q_t;
\]
in words, $\mathcal Q_{t_0}^- $ coincides with $\mathcal Q_{t_0} $ except that it ignores the last vertex reached by the process at time $t_0$, if there is one. 
First we show that, for all $t>0$,
\begin{equation}\label{eq:aux-Qminus}
\mathcal Q_t\geq \mathcal Q_{\theta_{\#A_t+1}}^- \quad\text{and}\quad h(e^{\beta t})\leq h\left(e^{\beta\theta_{\#A_t+1}}\right)\quad\P_{\mathcal{S}}\text{-almost surely}.
\end{equation}
The first inequality can be deduced from the following observations.
At any fixed time $t$ the quantity $\mathcal{Q}_{\theta_{\#A_t+1}}^-$ is the size of a longest red path (or red cluster) at time $\theta_{\#A_t+1}$, excluding the last occupied vertex.
Since $\theta_{\#A_t+1}$ is the first time that the occupied set has size $\#A_t+1$ then, when ignoring the last infected vertex, the occupied set has size $\# A_t$.
Thus, $\theta_{\#A_t+1}\geq t$.
However, between $t$ and $\theta_{\#A_t+1}$ only one vertex turns red (the one we ignore) while other red vertices might recover, then 
$\mathcal Q_t\geq \mathcal Q_{\theta_{\#A_t+1}}^-$.
The second inequality follows from the fact that $\theta_{\#A_t+1}\geq t$ and $h$ is non-decreasing by assumption.

By dividing both terms of the first inequality in \eqref{eq:aux-Qminus} by the terms in the second inequality and taking the liminf, since $\#A_t\rightarrow \infty$, we obtain
\[
\liminf_{t\rightarrow+\infty}\frac{\mathcal Q_t}{h(e^{\beta t})}\geq\liminf_{n\rightarrow+\infty}\frac{\mathcal Q_{\theta_n}^-}{h(e^{\beta\theta_n})}.
\]
The sub-sequence of times $(\theta_n )_n$ has some of the desired properties, but it is not independent of the process, thus we can not apply Lemma \ref{lemma_sub_seq} directly. 
To get round this problem we apply Proposition \ref{prop_split_property}. 
First, recall from Proposition \ref{prop_split_property} that $U$ is a discrete random variable representing a certain (finite) set of vertices. 
Therefore, it is enough to prove the claim on the event $\{U=F\}$ for any arbitrary finite set $F$, and then let $F$ vary. 
Fix a finite set of vertices $F$ such that $\Proba_\mathcal S(U=F)>0$, in particular we have $\#F>1$. 
From now on we reason under $\Proba_\mathcal S$ and conditional on the event $\{U=F\}$. 
For every $u\in F$, fix $u^*\in F$ such that $u^*\neq u$. For all $u\in F$ and $n\geq1$, define
\[
Y_n^u\coloneqq \tau_u+\theta_n^{(u)}\quad\text{and}\quad Z_n^u\coloneqq Y_n^u-\tau_{u^*}.
\]
All but finitely many vertices of $V$ belong to $G^{(u)}$ for some $u\in F$. 
Thus, by Proposition \ref{prop_split_property}, for all $n\geq 1$ large enough, for some $u\in F$ at time $\theta_n$ there will be a vertex $v$ in $G^{(u)}$ that will be reached by the process (see Figure \ref{fig:last_thm}). 
Therefore, $\P_{\mathcal{S}}$-almost surely given $\{U=F\}$ for all $n\geq1$ big enough there exists a (random) vertex $\mathbf{v}_n\in F$ and a (random) index $L_n\geq 1$ such that 
\[
\theta_n=Y_{L_n}^{\mathbf{v}_n}.
\]
By construction we have then
\begin{equation}\label{eq:aux-Qtheta}
\mathcal Q_{\theta_n}^-\geq \mathcal Q^{(\mathbf{v}_n^*)}_{Z_{L_n}^{\mathbf{v}_n}}.
\end{equation}
In fact, start by noticing that, since $\theta_n=Y_{L_n}^{\mathbf{v}_n} $, then $Z_{L_n}^{\mathbf{v}_n}= Y_{L_n}^{\mathbf{v}_n}-\tau_{\mathbf{v}_n^*} = \theta_n -\tau_{\mathbf{v}_n^*}$ is the time elapsed from the moment vertex $\mathbf{v}_n^*$ has been reached by the process.
Thus, the quantity $ Q^{(\mathbf{v}_n^*)}_{Z_{L_n}^{\mathbf{v}_n}}=Q^{(\mathbf{v}_n^*)}_{Y_{L_n}^{\mathbf{v}_n}-\tau_{\mathbf{v}_n^*}}$ is the size of a maximal red path (or maximal red cluster) at time $\theta_n$ when we restrict the observed process to $G^{(\mathbf{v}_n^*)}$.
However, the size of a maximal red path (or red cluster) in this case is at most the size of a maximal red path (or red cluster) at time $\theta_n$ in $G$.
Thus, 
\[
\mathcal Q_{\theta_n}\geq \mathcal Q^{(\mathbf{v}_n^*)}_{Z_{L_n}^{\mathbf{v}_n}}.
\]
Furthermore, the ``minus'' in the notation of \eqref{eq:aux-Qtheta} comes from the fact that at time $\theta_n$ only one vertex has been added to the cluster, and it was $\mathbf{v}_n\notin G^{(\mathbf{v}_n^*)}$, hence we can ignore the last added vertex, obtaining \eqref{eq:aux-Qtheta}.
Now observe that, by definition, $L_n\rightarrow \infty$ as $n\rightarrow \infty$.
Therefore, 
\[
\liminf_{n\rightarrow+\infty}\frac{\mathcal Q_{\theta_n}^-}{h(e^{\beta\theta_n})}\geq\min_{u\in F}\liminf_{n\rightarrow+\infty}\frac{\mathcal Q^{(u^*)}_{Z_n^u}}{h(e^{\beta Y_n^u})}, \quad
\P_{\mathcal{S}}\text{-almost surely, given } \{U=F\}.
\]
Next, fix $u\in F$. 
By Lemma \ref{lemma_controlled_growth}, $\P_{\mathcal{S}}$-almost surely 
$
\lim_n \frac{Z_n^u}{\log n} = \lim_n  \frac{Y_n^u}{\log n} = \lim_n \frac{\theta_n^{(u)}}{\log n} = \frac{1}{\boldsymbol{\alpha}}
$.
Moreover, the shifted process corresponding to $u^*$ has the same distribution as the original one and it is independent of $(Z_n^u)_{n\geq 1}$. 
Hence, by Lemma \ref{lemma_sub_seq} we get
\[
\liminf_{n\rightarrow+\infty}\frac{\mathcal Q^{(u^*)}_{Z_n^u}}{h(e^{\beta Y_n^u})}\geq 1, \quad
\P_{\mathcal{S}}\text{-almost surely, given } \{U=F\}. 
\]
By letting $F$ vary, we obtain the claim.
\end{proof}

\begin{figure}[ht]
    \centering
    \includegraphics[height=4.5cm]{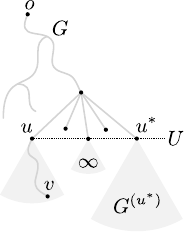}
    \caption{Illustration for the proof of Theorem \ref{thm_lower_bound}.}
    \label{fig:last_thm}
\end{figure}

\subsection{Proofs of Theorem \ref{thm_supercritical} and Corollary \ref{coro_lb_gw_at}}\label{subsect:GWproofs}

In this subsection we use the previous results to compute the sought asymptotic lower bounds.
The first two lemmas guarantee that the probabilities that we are investigating do indeed satisfy assumption \eqref{eq:assump-etah}.
Recall that for all $t<0$ we set $H_t\coloneqq  0$ and $M_t\coloneqq 0$.

\begin{Lemma}
\label{lemma_eta_bound_H}
Let $T\sim\mathrm{Exp}(1)$ be independent of the process, 
and $r\in(0,1)$. 
Then for all $x>e$ large enough,
\[
\Proba\left[H_{1-T}\leq\frac{r\log x}{\log\log x}\right]\leq e^{-1/x}.
\]
\end{Lemma}

\begin{proof}
Fix $x>e$ and define
\begin{equation}\label{eq:mx-and-px}
m_x\coloneqq \frac{r\log x}{\log\log x}\quad\text{and}\quad p_x:\R\rightarrow[0,1],~t\mapsto\Proba[H_t\leq m_x].
\end{equation}
Next, denote
\[
\mathcal M(x)\coloneqq \inf\left\lbrace 1-p_x(t),~t\in\left[1-\log 2,1\right]\right\rbrace.
\]
By standard facts we have that

\begin{align*}
 \Proba[H_{1-T}\leq m_x]  &= \E[p_x(1-T)\mathbf{1}_{T>\log 2}]+\E[p_x(1-T)\mathbf{1}_{T\leq\log 2}] \\
& \leq \P (T>\log 2) +\Proba[T\leq\log 2]\sup_{t\in [1-\log 2, 1]}p_x(t)
 = 1-\frac{\mathcal{M}(x)}{2}.
\end{align*}
\
Now, using the fact that $\log(1+z)\leq z$ for all $z>-1$, taking logarithms we obtain
\begin{equation}
    \label{eq_log_eta_H}
    \log\Proba[H_{1-T}\leq m_x] \leq \log \left ( 1-\frac{\mathcal M(x)}{2}\right )\leq-\frac{\mathcal{M}(x)}{2}.
\end{equation}
Now, for all $n\geq0$ we let the vertex $v_n$ correspond to the left-most individual of the $n$-th generation in the tree, and recall that $(\mathbf{C}_v)_{v\in V}$ is the sequence of recovery clocks.
Moreover, for all $v\in V$ let $B_v$ denote the number of children of vertex $v$. 
Set $T_o\coloneqq 0$ for convenience and notice that, for all $t\in[1-\log2, 1]$,
\[
\bigcap_{k=0}^{\lfloor m_x\rfloor}
\left\lbrace \mathbf{C}_{v_k}>1,~ B_{v_k}\geq1,~ T_{v_k}<\frac{1-\log 2}{m_x+1} \right\rbrace\subset\{H_t>m_x\}.
\]
Therefore, considering the corresponding probabilities,
\begin{equation}\label{eq:aux-Mhat}
\mathcal M(x)\geq \Proba[B\geq 1]^{m_x+1}e^{-\gamma(m_x+1)}\left(1-\exp\left[-\frac{1-\log 2}{m_x+1}\right]\right)^{m_x}.
\end{equation}
For the remainder of this proof, let 
\begin{equation}\label{eq:Mhat}
\widehat{\mathcal M}(x)\coloneqq  \Proba[B\geq 1]^{m_x+1}e^{-\gamma(m_x+1)}\left(1-\exp\left[-\frac{1-\log 2}{m_x+1}\right]\right)^{m_x}.
\end{equation}
Next, since $\log(1-e^{-z})\sim\log z$ as $z\rightarrow0^+$ together with \eqref{eq:mx-and-px}, we have
\[
\log\widehat{\mathcal M}(x)\sim m_x \log \left ( \frac{1-\log 2}{m_x+1}\right )\sim -m_x\log m_x\sim -r\log x\quad\text{as}\quad x\rightarrow+\infty.
\]
Thus, since $r<1$ we have that for all $x$ large enough
\[
\widehat{\mathcal M}(x)\geq 2\exp(-\log x).
\]
Finally, since $\widehat{\mathcal M}(x)\leq\mathcal M(x)$ (by \eqref{eq:Mhat} and \eqref{eq:aux-Mhat}), by plugging this inequality into (\ref{eq_log_eta_H}) we get
\[
\begin{split}
\log\Proba[H_{1-T}\leq m_x] \leq -\frac{\mathcal M(x)}{2} \leq -\frac{\widehat{\mathcal M}(x)}{2} \leq - \frac{1}{2}\, \frac{2}{x}=-\frac{1}{x}.
\end{split}
\]
By taking exponentials, we obtain the claimed result.
\end{proof}

\begin{Lemma}
\label{lemma_eta_bound_M}
Let $T\sim\mathrm{Exp}(1)$ be independent of the process 
and $r\in(0,1)$. Then for all $x>e^e$ large enough we have
\[
\Proba\left[M_{1-T}\leq\frac{r\log x}{\log\log\log x}\right]\leq e^{-1/x}.
\]
\end{Lemma}

\begin{proof}
We proceed as in the proof of Lemma \ref{lemma_eta_bound_H}. 
Fix $x>e^e$ and define
\[
m_x'\coloneqq \frac{r\log x}{\log\log\log x}\quad\text{and}\quad  p_x':\R\rightarrow[0,1], \ t\mapsto\Proba[M_t> m_x].
\]
Thus, by letting
$ \mathcal M'(x)\coloneqq \inf\left\lbrace p_x'(t),~t\in\left[1-\log 2,1\right]\right\rbrace $, we have
$ \log\Proba[M_{1-T}\leq m_x]\leq-\mathcal M'(x)/2$.

Let $\mathcal{V}_x$ denote the set of vertices corresponding to the first $1+\lfloor m_x'\rfloor$ individuals of the (deterministic) \emph{binary tree}. 
Now, since $G$ is supercritical, we can reason as follows.
As in the previous proof, for $v\in V$ let $B_v$ denote the number of children of $v$.
The probability that $G$ contains $\mathcal{V}_x$ is positive and, since the variables $\{B_v\}_v$ are i.i.d., it can be bounded by
\[
\P \left [ \bigcap_{v\in \mathcal{V}_x} \{B_v \geq 2\}\right ] = (1-\zeta(0)-\zeta(1))^{1+\lfloor m_x'\rfloor},
\]
where, as defined at the beginning of Section \ref{sect:GWtree}, $\zeta(k)=\P(B_v=k)$.
Furthermore, by definition of $\mathcal{V}_x$, we have that all such individuals are to be found within generation $\log_2(1+\lfloor m_x'\rfloor)$.

Now set $T_o\coloneqq 0$ for convenience.
For all $t\in[1-\log 2, 1]$, we have
\[
\bigcap_{v\in \mathcal{V}_x}\left\lbrace \mathbf{C}_v>1,~ B_v \geq 2,~T_v<\frac{1-\log 2}{ \log_{2}(m_x'+1)} \right\rbrace\subset\{M_t>m'_x\}.
\]
Therefore computing the probability,
\[
\mathcal M'(x)\geq \Proba[B_o \geq 2]^{m_x'+1}e^{-\gamma(m_x'+1)}\left(1-\exp\left[-\frac{1-\log 2}{ \log_{2}(m_x'+1)}\right]\right)^{m_x'}=:\widehat{\mathcal M}'(x).
\]
Next, by reasoning like in the previous lemma, we obtain (for all fixed $r<1$)
\[
\log \widehat{\mathcal M}'(x)\sim -m_x'\log\log m_x'\sim-r\log x\quad\text{as}\quad x\rightarrow+\infty,
\]
which implies the claim.
\end{proof}

Now we finally able to provide a proof of Theorem \ref{thm_supercritical}.

\begin{proof}[Proof of Theorem \ref{thm_supercritical}]
It suffices to show that, almost surely on $\mathcal S$, 
\[
\liminf_{t\rightarrow+\infty}\frac{H_t\log t}{t}\geq \boldsymbol{\alpha}\quad\text{and}\quad\liminf_{t\rightarrow+\infty}\frac{M_t\log\log t}{t}\geq \boldsymbol{\alpha}.
\]
We start with the first claim.
Fix $r\in (0,1)$ and for all $x\in (e,\infty)$ define $h(x)\coloneqq  \frac{r\log x}{\log\log x} $.
This function is non-decreasing and by Lemma \ref{lemma_eta_bound_H}, it satisfies the assumptions of Theorem \ref{thm_lower_bound} for $(\mathcal Q_t)_{t\in\R} \equiv (H_t)_{t\in\R}$ and thus for all $\beta\in(0, \boldsymbol{\alpha})$ we have
\[
\liminf_{t\to+\infty}\frac{H_t}{h(e^{\beta t})}\geq1,\quad\P_\mathcal{S}\text{-almost surely}.
\]
Simplifying the above expression we get $ \liminf_{t}\frac{H_t\log t}{t}\geq r\beta$, $\P_\mathcal{S}$-almost surely.
Then, by letting $\beta\rightarrow \boldsymbol{\alpha}$ and $r\rightarrow1$ we get the first part of the claim. 
The second part of the claim is obtained similarly, by using Lemma \ref{lemma_eta_bound_M} and Theorem \ref{thm_lower_bound} with $h$ so that, for all $x\in (e^e,\infty)$, $h(x)= \frac{r\log x}{\log\log\log x} $ and $(\mathcal Q_t)_{t\in\R} \equiv (M_t)_{t\in\R}$.
\end{proof}

\begin{Remark}\label{rem:auxiliary-new-theorem}
Observe that $\P_\mathcal S (\#A_t\rightarrow \infty)=1$,  and thus by Lemma \ref{lemma_controlled_growth},
\[
\theta_{\#A_t}\sim\frac{\log \#A_t}{\boldsymbol{\alpha}}\sim\frac{\log(\#A_t+1)}{\boldsymbol{\alpha}}\sim\theta_{\#A_t+1}.
\]
Then, since for all $t>0$ we have $\theta_{\#A_t}\leq t\leq\theta_{\#A_t+1}$, we deduce that for all $t$ large enough,
\begin{equation}\label{eq:aux-logAt}
\log\#A_t\sim \boldsymbol{\alpha} t,\quad \P_\mathcal{S}\text{-almost surely}.
\end{equation}
\end{Remark}

Finally we can provide a proof of Corollary \ref{coro_lb_gw_at}.

\begin{proof}[Proof of Corollary \ref{coro_lb_gw_at}]
By Remark \ref{rem:auxiliary-new-theorem} 
when $t\to \infty$ we obtain
\[
\frac{\log (\#A_t)}{\log \log (\#A_t)} \stackrel{\eqref{eq:aux-logAt} }{\sim} \frac{\boldsymbol{\alpha} t}{\log (\boldsymbol{\alpha} t)}
\sim \frac{\boldsymbol{\alpha} t}{\log t}, \quad \P_\mathcal{S}\text{-almost surely}.
\]
Finally, the result follows from Theorem \ref{thm_supercritical}.
\end{proof}
%
%

\section{Proof of Proposition \ref{prop:blue}}\label{section:blue}

The first part of the claim follows from Remark \ref{rem:auxiliary-new-theorem}.
In fact, since for all $t\geq 0$ we have $M_t\leq\#A_t$, the asymptotic given in (\ref{eq:aux-logAt}) implies that for all $\delta>0$, $\P_\mathcal{S}-$almost surely for all $t$ big enough
\[
\log M_t\leq \left(\boldsymbol{\alpha}+\frac{\delta}{2}\right)t.
\]
The remaining part of the proof of Proposition \ref{prop:blue} is split into Lemmas \ref{lemma:theta-t} and \ref{lemma:gammaTheta} below. 
In the following, let $\mathcal{B}_x(r)$ denote the ball (in the graph metric) of radius $r$ centered at $x$.

\begin{Lemma}\label{lemma:theta-t}
In the same setting as above there is a constant $\tilde{c}>0$ large enough, so that
\[
\P \left ( \limsup_{t\to+\infty} \bigl \{ A_t \subset \mathcal{B}_o(\bar{c}t)\bigr \} \right ) =1,
\]
and in particular for all $\gamma>0$,
$
\P \left ( \limsup_t \frac{ H_{t}}{t} \leq \bar{c} \right ) =1.
$
\end{Lemma}
\begin{proof}
To get the first part, we investigate what happens at integer times and then use the fact that $(A_t)_{t\geq 0}$ is non-decreasing. Let $\tilde{c}>1$ be a constant so large, that 
\begin{equation}\label{eq:cond-on-ctilde}
1+\tilde{c} \log \tilde{c} -\tilde{c}(1+\log \m)>0.
\end{equation}
Since 
$
\P \left (  A_n \not \subset \mathcal{B}_o(\tilde{c} n) \right ) 
= \E \left [ \P \left (  A_n \not \subset \mathcal{B}_o(\tilde{c} n) \mid G \right ) \right ] 
$,
we start by bounding the inner part of the expectation.
In the following, let $G_k$ denote generation $k$ of the tree $G$, thus
\[
\begin{split}
& \P \left (  A_n \not \subset \mathcal{B}_o(\tilde{c} n) \mid G \right ) 
= \P \left ( \exists x \in G_{\lfloor \tilde{c}n \rfloor +1} \ : \ \tau_x \leq n \mid G \right )
%
\leq \# G_{\lfloor \tilde{c}n \rfloor +1} \P(\text{Poi}(n) > \tilde{c}n),
\end{split}
\]
where the second-to-last inequality follows from a union bound and the following fact.
The probability that a vertex at distance $\ell$ from $o$ is reached by time $n$ coincides with the probability that within time $n$ one witnesses at least $\ell$ observations of a Poisson process of rate $1$.
In this case we have $\ell=\lfloor \tilde{c}n \rfloor +1>\tilde{c}n$.

By taking the expectation and recalling that $\m>1$ denotes the mean of the offspring distribution, the above calculations imply
\[
P \left (  A_n \not \subset \mathcal{B}_o(\tilde{c} n) \right )  \leq \m^{\lfloor \tilde{c}n \rfloor +1}\P(\text{Poi}(n) > \tilde{c}n).
\]
Now we reason as follows.
For any $\lambda>0$, Markov's inequality gives
\[
\begin{split}
& \P(\text{Poi}(n) > \tilde{c}n) = \P(e^{\lambda\text{Poi}(n)} > e^{\lambda \tilde{c}n})
\leq \frac{\E \left [ e^{\lambda\text{Poi}(n)}\right ]}{e^{\lambda \tilde{c}n}} = \exp \left \{-n \left (1+\tilde{c}\lambda-e^\lambda \right) \right \}.
\end{split}
\]
Since $\tilde{c}>1$ we can choose $\lambda=\log \tilde{c}$, obtaining
\[
P \left (  A_n \not \subset \mathcal{B}_o(\tilde{c} n) \right )  \leq \m \exp \left \{-n \left [1+\tilde{c} \log \tilde{c} -\tilde{c}(1 +\tilde{c}\log \m) \right] \right \}.
\]
By defining $\mathbf{c}_1\coloneqq  1+\tilde{c} \log \tilde{c} -\tilde{c}(1 +\log \m)$, by \eqref{eq:cond-on-ctilde} we obtain
\[
\sum_{n\geq 0}\P \left ( A_n \not \subset \mathcal{B}_o(\tilde{c}  n)\right ) \leq \m \sum_{n\geq 0}e^{-\mathbf{c}_1 n}<+\infty.
\]
Therefore, by the first Borel-Cantelli lemma, almost surely for all $n\in\N$ large enough we have $A_n \subset \mathcal{B}_o(\tilde{c} n)$ and in particular, fixing $\bar c>\tilde c$, for all $t\geq 0$ large enough we get
\[
A_t\subset A_{\lceil t\rceil}\subset \mathcal{B}_o(\tilde{c} \lceil t\rceil)\subset \mathcal{B}_o(\tilde{c} (t+1))\subset\mathcal{B}_o(\bar{c}t).
\]
Finally, since for all $t\geq 0$ $H_t$ cannot be larger than the maximal distance of any vertex in $A_t$ from the origin, we obtain the claimed result.
\end{proof}

Now, assuming that $G$ has bounded degree, we shall show that for all fixed $\varepsilon>0$, if the recovery rate $\gamma$ is large enough, then asymptotically as $t\to \infty$ the volume of the largest red cluster at time $t$ is at most $\varepsilon t$.
Recall that $\Delta$ denotes the maximal degree of $G$.

\begin{Lemma}\label{lemma:gammaTheta} 
Fix $\varepsilon>0$ arbitrarily small. Then there is a critical value $\gamma_c = \gamma_c(\Delta, \varepsilon, \tilde{c})$ so that for all $\gamma>\gamma_c$,
\[
\P \left ( \limsup_{t\to+\infty} \frac{M_{t}}{t}\leq \varepsilon \right )=1.
\]
\end{Lemma}
\begin{proof}
In this proof we shall use a percolation argument.
Since the maximal degree of $G$ is $\Delta$, then $G$ is a subgraph of the regular tree of degree $\Delta$.
We denote by $\mathbb{T}_{\Delta}$ such a tree.
Furthermore, for all $t\geq 0$ set 
\[
\mathbb{T}_{\Delta}(t)\coloneqq \{x\in \mathbb{T}_{\Delta} \ : \ d(o,x)\leq \bar{c}t\},
\]
i.e., the set of nodes of $\mathbb{T}_{\Delta}$ up to generation $\lfloor\bar{c}t\rfloor$.

Since the minimum of $\Delta$ independent Exp($1$) random variables is distributed as an Exp($\Delta$), for all $v\in V$ the probability that $v$ activates a seed at any of its neighbors before recovering is
\[
\P\left(\mathbf{C}_v>\min_{w\in\partial^*\{v\}} T_{w}\right) \leq \frac{\Delta}{\gamma + \Delta},
\]
where the minimum is taken over all vertices $w$ that are (direct) descendants of $v$.
Now define
\[
p(\gamma)\coloneqq \frac{\Delta}{\gamma + \Delta},
\]
and consider Bernoulli site percolation on $\mathbb{T}_{\Delta}$ of parameter $p(\gamma)$. Clearly, by taking $\gamma$ large enough, we can make $p(\gamma)$ as small as we wish. Let $\gamma$ be so large that we are in the sub-critical regime (i.e. \ there is almost surely no infinite open cluster) and such that if we set
\[
\kappa\coloneqq p(\gamma)\bigl (1-p(\gamma) \bigr )^{\Delta-1}\frac{\Delta^{\Delta}}{(\Delta-1)^{\Delta-1}},
\]
then we have
\begin{equation}\label{eq:conditionONgamma}
\frac{\Delta \, \bar{c}}{\log_{\Delta}(1/\kappa)} \leq\varepsilon.
\end{equation}
At this point, let $\mathcal{K}_t$ denote the size of largest open cluster containing a vertex of $\mathbb{T}_{\Delta}(t)$.
Then the main tool towards the sought conclusion is \cite[Theorem 2]{Ariascastro} which states that 
\begin{equation}\label{eq:Ariascastro}
\P \left ( \frac{\mathcal{K}_t}{\lfloor\bar{c}t\rfloor}\to \frac{1}{\log_{\Delta}(1/\kappa)} \right )=1.
\end{equation}
We now define a coupling between the above site percolation process and FPPHE with recovery as follows. 
Since the probability that any vertex of $G$ activates a seed at any of its neighbors is at most $p(\gamma)$, we can declare each vertex that does so to be \emph{open}.

For $t\geq 0$, let $M'_t$ denote the size of the largest red cluster in $A_t\setminus\partial A_t$, where $\partial A_t$ denotes the internal boundary of $A_t$ (i.e., all $v\in A_t$ that have a neighbor outside $A_t$). 
Since the degree of $G$ is at most $\Delta$, one has $M_t\leq\Delta M'_t+1$. 
The ``+1'' is needed if $M'_t=0$. 
In fact, it might happen that red vertices are only present on $\partial A_t$, since they would be isolated red vertices, the largest red cluster has cardinality $1$.
Moreover, since $\mathcal{B}_o(\tilde{c} (t+1)) \subset \mathbb{T}_{\Delta}(t)$, then by Lemma \ref{lemma:theta-t} we have that almost surely $A_t\subset \mathbb{T}_{\Delta}(t)$ for all $t$ large enough, thus $M'_t\leq \mathcal{K}_t$ by construction of the coupling. Summing up, almost surely for all $t$ large enough,
$ M_t \leq \Delta \,\mathcal{K}_t+1 $.

Finally, this implies the claim, as 
$\frac{\Delta\,\mathcal{K}_t+1}{t} \to \frac{\Delta \, \bar{c}}{\log_{\Delta}(1/\kappa)} \leq\varepsilon$ almost surely by \eqref{eq:Ariascastro} and \eqref{eq:conditionONgamma}.
\end{proof}

\appendix
\section{Alternative proof of Theorem \ref{thm_critical}(ii)}\label{sect:appendix}

In Subsection \ref{subsect:line-liminf} we presented a short proof for the liminf of Theorem \ref{thm_critical}.
However, the result can be obtained via another proof that we believe to be of independent interest, hence we present it below.

We need to introduce some notation.

\begin{Definition}
\label{def_complete_recovery}
For all $n\geq0$ and $v\in\range{1}{n}$, define the events
\[
E_v^n\coloneqq \left\lbrace \mathbf{C}_v>\sum_{k=v+1}^{n+1} T_k\right\rbrace,
\]
namely, $E_v^n$ corresponds to the event that $v$ does not recover before the reaching time of vertex $n+1$.
Moreover, define the \emph{complete recovery probability of order $n$} as the probability that all vertices in $\range{1}{n}$ have recovered before time $\tau_{n+1}$ i.e.,
\[
\nu_n\coloneqq \P \left ( \bigcap_{v=1}^n ({E_v^n})^c \right )=1-\Proba\left[\bigcup_{v=1}^n {E_v^n}\right].
\]
\end{Definition}

Recall the definition of $R_t$ from \eqref{eq:def-R_t}; then, $E_v^n=\{v\in R_{\tau_{n+1}}\}$. 
As a consequence, for all $n\geq0$,
\[
\nu_n=\Proba(R_{\tau_{n+1}}=\{n+1\}).
\]
The next lemma gives a formula that will be used to obtain the asymptotic behavior of $(\nu_n)_{n\geq 1}$. 

\begin{Lemma}
\label{lemma_complete_recovery}
For all $n\geq 1$ and $\ell\geq 1$, let
\[
N_\ell^n\coloneqq \left\lbrace\mathbf x\in ({\N^*})^\ell:\sum_{k=1}^\ell \mathbf x_k\leq n\right\rbrace \quad \text{and}\quad S_\ell(n)\coloneqq \sum_{\mathbf x\in N_\ell^n}\prod_{k=1}^\ell(1+k\gamma)^{-\mathbf x_k}.
\]
Then for all $n\geq 1$ we have
\begin{equation}
    \label{eq_nu_n}
    \nu_n=1+\sum_{\ell=1}^n (-1)^\ell S_\ell(n).
\end{equation}
\end{Lemma}
In words, $N_\ell^n$ is the set of all vectors of length $\ell$, so that the sum of the entries is at most $n$.
$S_\ell(n)$ is a quantity related to the probability that all vertices in any set of cardinality $\ell \in \{1, 2, \ldots , n\}$ recover before $\tau_{n+1}$.
Its meaning will become clear in the proof.

\begin{proof}[Proof of Lemma \ref{lemma_complete_recovery}]
Let $n\geq 1$, and for $k\in \{1, \ldots ,n\} $ let $\mathcal{B}_k$ denote any subset of $\range{1}{n}$ of cardinality $k$; for $k=0$ set $\mathcal{B}_0\coloneqq  \emptyset$.
Then, by the inclusion-exclusion principle applied to $\Proba\left[\bigcup_{v=1}^n {E_v^n}\right]$,
\[
\nu_n = 1-\Proba\left[\bigcup_{v=1}^n {E_v^n}\right] = 1- \sum_{k=1}^n (-1)^{k+1}\, \P \left[\bigcap_{v\in \mathcal{B}_k} E_v^n\right] = 1+\sum_{k=1}^n (-1)^{k}\, \P \left[\bigcap_{v\in \mathcal{B}_k} E_v^n\right].
\]
For simplicity, 
we re-write the above as 
\begin{equation}\label{formula_nu_n}
\nu_n=1+\sum_{A\subset\range{1}{n},~A\neq \emptyset}(-1)^{\# A} \,  \Proba\left[\bigcap_{v\in A} E_v^n\right].
\end{equation}
Now fix any subset $A\subset\range{1}{n}$, define $\ell\coloneqq \#A\geq 1$  and let $(a_i)_{i\in\range{1}{\ell }}$ denote the elements of $A$ in \emph{increasing} order. 
(Note that $A$ is not necessarily connected.)
For convenience, define $a_{\ell +1}\coloneqq n+1$.

For all $\ell\geq 1$ (and hence for all $A\subset\range{1}{n}$) fixed, we now compute $\P \left[ \bigcap_{v\in A} E_v^n \right]$.
Start by observing that by independence of the recovery clocks $\{\mathbf{C}_j \}_j$ we have
\begin{equation}\label{eq:aux-clocks}
\P \left ( \bigcap_{v\in A} E_v^n \mid \{T_k\}_k\right ) = \prod_{j=1}^\ell \P \left ( \mathbf{C}_{a_j}>\sum_{k=a_j+1}^{n+1} T_k \mid \{T_k\}_k\right ) 
= \prod_{j=1}^\ell \exp\left(-\gamma\sum_{k=a_j+1}^{n+1} T_k\right).
\end{equation}
Therefore,
\[
\begin{split}
\P \left[ \bigcap_{v\in A} E_v^n \right] & 
= \E \left [ \P \left ( \bigcap_{v\in A} E_v^n \mid \{T_k\}_k\right ) \right ] 
\ \stackrel{\eqref{eq:aux-clocks} }{=} \ \E\left[\prod_{j=1}^\ell \exp\left(-\gamma\sum_{k=a_j+1}^{n+1} T_k\right)\right]\\
& = \E\left[\exp\left(-\gamma\sum_{j=1}^{\ell} \Bigl ( j\sum_{k=a_j+1}^{a_{j+1}} T_k\Bigr ) \right)\right] 
= \prod_{j=1}^\ell \prod_{k=a_j+1}^{a_{j+1}}\int_{\R_{\geq 0}}e^{-\gamma jt}e^{-t}\mathrm dt ,
\end{split}
\]
where the last equality follows from independence of the recovery clocks and of the passage times.
By solving the integral one obtains
\begin{equation}\label{eq:P-cap}
\P \left[ \bigcap_{v\in A} E_v^n \right]  =\prod_{j=1}^\ell  (1+j\gamma)^{-(a_{j+1}-a_j)} .
\end{equation}
Next, we define a function $\mathbf\Psi_\ell$ from $N_\ell^n$ to 
the collection of all subsets of cardinality $\ell$ contained in $\range{1}{n}$ as follows.
For any fixed $\mathbf{x}=(\mathbf{x}_1, \ldots , \mathbf{x}_{\ell})\in N_\ell^n$ let 
\[
\mathbf\Psi_\ell (\mathbf x) \coloneqq  \left \{n+1-\mathbf{x}_1, ~ n+1-\mathbf{x}_1-\mathbf{x}_2, ~\ldots , ~ n+1-\sum_{i=1}^\ell \mathbf{x}_i \right \}.
\]
This defines a one-to-one correspondence.
By construction, for all 
$\mathbf x\in N_\ell^n$ equation \eqref{eq:P-cap} gives
\[
\P \left[ \bigcap_{v\in \mathbf\Psi_\ell(\mathbf x)} E_v^n \right] =\prod_{k=1}^\ell (1+k\gamma)^{-\mathbf x_k}.
\]
Finally, since $\mathbf\Psi_\ell $ defines a one-to-one correspondence, we obtain \eqref{eq_nu_n} by replacing $A\subset\range{1}{n}$ by $\mathbf\Psi_\ell (\mathbf x)$ in (\ref{formula_nu_n}) for every $\ell \in\range{1}{n}$ and $\mathbf x\in N_\ell^n$.
\end{proof}

The next result shows that $(\nu_n)_{n\geq1}$ converges to a positive value, which we compute explicitly.

\begin{Proposition}
\label{prop_limit_nu_n}
The sequence $(\nu_n)_{n\geq 1}$ converges to $e^{-1/\gamma}$.
\end{Proposition}

\begin{proof}
Fix an arbitrary $\ell \geq 1$ and recall the definition of $S_\ell(n)$ from Lemma \ref{lemma_complete_recovery}.
Define
\[
S_\ell^{\infty}\coloneqq \sum_{ \mathbf x\in(\N^*)^\ell }\prod_{k=1}^\ell(1+k\gamma)^{-\mathbf x_k}.
\]
Then, since $S_\ell(n)$ is a series with positive terms, we have $S_\ell(n)\rightarrow S_\ell^{\infty}$ as $n\to \infty$.
Observe that since every entry of $\mathbf{x}$ varies in $\N^*$, we have
\[
\sum_{ \mathbf x\in(\N^*)^\ell }\prod_{k=1}^\ell(1+k\gamma)^{-\mathbf x_k}=\prod_{k=1}^\ell\sum_{j=1}^{+\infty}\left(\frac{1}{1+k\gamma}\right)^j,
\]
which gives that for all fixed $\ell\geq 1$, one has  $S_\ell^{\infty} = \prod_{k=1}^\ell\frac{1}{k\gamma}=\frac{1}{\ell!\gamma^\ell} $.

Now note that for all fixed $n\geq 1$ and $\ell \geq 1$ we have $S_\ell(n)\leq S_\ell^\infty $ (which is summable, as $\sum_{\ell\geq 0} S_\ell^{\infty}=e^{1/\gamma}$) then, by Lemma \ref{lemma_complete_recovery} and the dominated convergence theorem with respect to the counting measure, we deduce that
\[
\lim_{n\rightarrow \infty}\nu_n
=\lim_{n\rightarrow \infty}\sum_{\ell\geq 0}(-1)^\ell S_\ell(n)
=\sum_{\ell\geq 0}(-1)^\ell S_\ell^\infty
=\sum_{\ell\geq 0}\frac{(-1/\gamma)^\ell}{\ell!}=e^{-1/\gamma},
\]
as claimed.
\end{proof}

The result below is an intermediate step, as it gives a liminf over a sub-sequence of times, namely the random times $(\tau_n)_{n\geq 1}$.
Clearly, $\liminf_{n\rightarrow \infty} H_{\tau_n}\geq \liminf_{t\rightarrow \infty} H_t$.


\begin{Theorem}
\label{thm_recovery_io}
We have  $\P(\{R_{\tau_n}=\{n\}\}, \text{ i.o.})=1$.
\end{Theorem}

\begin{proof}
A straightforward consequence of Proposition \ref{prop_limit_nu_n} and Fatou's lemma is 
\begin{equation}\label{eq:P(E1)gamma}
\P \left ( \limsup_{n\to+\infty} \, \{R_{\tau_n}=\{n\}\} \right ) \geq  \limsup_{n\to+\infty} \P(R_{\tau_n}=\{n\})  
= \lim_{n\to+\infty} \nu_n = e^{-1/\gamma}>0.
\end{equation}
We shall use this fact to deduce the statement from the Kolmogorov 0-1 law.
Now we condition on the event that for all $k\geq 1$, $T_{k+1}<+\infty$ and $\mathbf{C}_k<+\infty$ as well as $\tau_n\to+\infty$; since this occurs almost surely (cf.\ Lemma \ref{lemma_fpp}), it is not a real restriction. 
For all $m\geq 1$ let
\[
\mathcal{T}_m \coloneqq \sigma \left ( (\mathbf{C}_j)_{j\geq m}, \, (T_{j+1})_{j\geq m} \right ) \quad \text{and} \quad \mathcal{T}\coloneqq  \bigcap_{m\geq 1} \mathcal{T}_m
\]
denote the $\sigma$-algebra generated by all recovery clocks and passage times starting from vertex $m$ and the corresponding tail-$\sigma$-algebra, respectively. Then observe that for all $k\geq 1$ we have
\begin{equation}\label{eq:limsup_shift}
\limsup_{n\to+\infty} \, \{R_{\tau_n}=\{n\}\}=\limsup_{n\to+\infty} \, \{R_{\tau_n}\setminus\range{1}{k-1}=\{n\}\}.
\end{equation}
This follows from the fact that, for all fixed $k$, the first $k-1$ vertices will eventually recover.
(Recall that we are working under the assumption that all recovery times are finite.) 
Formally, denoting
\[
t_0\coloneqq \max_{i\in\range{1}{k-1}}(\tau_i+\mathbf C_i),
\]
we have $\tau_n>t_0$ for all $n$ large enough since $t_0<+\infty$ and $\tau_n\to+\infty$. 
Therefore since by construction after time $t_0$ the first $k-1$ vertices are recovered, this implies that $R_{\tau_n}=R_{\tau_n}\setminus\range{1}{k-1}$ for all $n$ large enough, hence  (\ref{eq:limsup_shift}) follows. 

Now, observe that the RHS of (\ref{eq:limsup_shift}) is $\mathcal T_k$-measurable for all $k$.
In fact, for all $n\geq k$ the set $R_{\tau_n}\setminus\range{1}{k-1}$ corresponds to the set of red vertices for the process \emph{started} at $k$ when the occupied region reaches size $n-k+1$. 
This implies that the LHS of \eqref{eq:limsup_shift} is $\mathcal T$-measurable and therefore the result follows by Kolmogorov 0-1 law and \eqref{eq:P(E1)gamma}.
\end{proof}

Finally we are ready to prove part (ii) of Theorem \ref{thm_critical}.

\begin{proof}[Proof of Theorem \ref{thm_critical}(ii)]
Now we will show that, almost surely, $\liminf_{t\rightarrow \infty}H_t=0$.
Start by observing that, by definition of the process and the properties of the exponential distribution,
\begin{equation}\label{eq:not-simult}
\P( \exists \, (v,w) \in \range{1}{n}^2 \ : \ \tau_v = \mathbf{C}_w + \tau_w)=0.
\end{equation}
Moreover, by Theorem \ref{thm_recovery_io} almost surely there is an infinite increasing sub-sequence of indices $(n_k)_{k\in\N}$ for which $ R_{\tau_{n_k}}=\{n_k\} $ for all $k\in\N$.
Since the process takes place in continuous time, by \eqref{eq:not-simult} for each $k\geq 1$ there is a value $\varepsilon_k\in (0,1)$ so that
$
R_{\tau_{n_k}-\varepsilon_k}=\{n_k\}\setminus \{n_k\} = \emptyset$.
(In fact, at time $\tau_{n_k}$ vertex $\{n_k\}$ is the only one to be red and thus, right before, no vertex can be red.)
Hence the result.
\end{proof}

\paragraph{Acknowledgments.}
E.C.\ acknowledges partial support by ``INdAM--GNAMPA, Project code CUP$\_$ E53C22001930001''.
This project started when T.G.S.\ was visiting the department of Mathematics and Physics at the University of Roma Tre for an internship, which was supported by the Erasmus+ program of the European Union and the Mathematics department of the ENS.
The authors are grateful to an anonymous referee for their helpful suggestions that considerably helped improving the paper.

\nocite{}
\bibliographystyle{amsalpha}
\bibliography{bibliography}

\end{document}